\documentclass[smallextended]{svjour3}       % onecolumn (second format)
\usepackage{mathtools}
\usepackage[normalem]{ulem} % to use \sout

\usepackage{amssymb}
\usepackage{mathrsfs}
\usepackage{color}

\newcommand{\cut}[1]{{}}
\newcommand{\vb}{{\mathbf{b}}}

\newcommand{\vq}{{\mathbf{q}}}

\newcommand{\vs}{{\mathbf{s}}}
\newcommand{\vt}{{\mathbf{t}}}

\newcommand{\vx}{{\mathbf{x}}}
\newcommand{\vy}{{\mathbf{y}}}
\newcommand{\vz}{{\mathbf{z}}}

\newcommand{\vA}{{\mathbf{A}}}

\newcommand{\vD}{{\mathbf{D}}}

\newcommand{\vI}{{\mathbf{I}}}

\newcommand{\vM}{{\mathbf{M}}}

\newcommand{\vR}{{\mathbf{R}}}

\newcommand{\PD}{{\mathbf{T}_{\textnormal{PD3O}}}}
\newcommand{\vT}{{\mathbf{T}}}

\newcommand{\cL}{{\mathcal{L}}}

\newcommand{\cS}{{\mathcal{S}}}

\newcommand{\cX}{{\mathcal{X}}}

\newcommand{\vzero}{\mathbf{0}}

\newcommand{\prox}{\mathbf{prox}}

\DeclareMathOperator*{\argmin}{arg\,min}
\newcommand{\bc}{\begin{center}}
\newcommand{\ec}{\end{center}}

\newcommand{\bdm}{\begin{displaymath}}
\newcommand{\edm}{\end{displaymath}}

\newcommand{\beq}{\begin{equation}}
\newcommand{\eeq}{\end{equation}}

\newcommand{\bfl}{\begin{flushleft}}
\newcommand{\efl}{\end{flushleft}}

\newcommand{\bt}{\begin{tabbing}}
\newcommand{\et}{\end{tabbing}}

\newcommand{\beqn}{\begin{align}}
\newcommand{\eeqn}{\end{align}}

\newcommand{\beqs}{\begin{align*}} % no equation numbers
\newcommand{\eeqs}{\end{align*}}  % no equation numbers

\newtheorem{assumption}{Assumption}

\smartqed  % flush right qed marks, e.g. at end of proof
\usepackage[colorlinks=true,breaklinks=true,bookmarks=true,urlcolor=blue,
     citecolor=blue,linkcolor=blue,bookmarksopen=false,draft=false]{hyperref}
\begin{document}

\title{A new primal-dual algorithm for minimizing the sum of three functions with a linear operator\thanks{The work was supported in part by the NSF grant DMS-1621798.}
}

\titlerunning{PD3O: primal-dual three-operator splitting}        % if too long for running head

\author{Ming Yan}


\institute{M. Yan \at
              Department of Computational Mathematics, Science and Engineering;  
							Department of Mathematics; 
							Michigan State University, East Lansing, MI, USA, 48864.
							\email{myan@msu.edu}
}

\date{Received: date / Accepted: date}

\maketitle

\begin{abstract}
In this paper, we propose a new primal-dual algorithm for minimizing $f(\vx)+g(\vx)+h(\vA\vx)$, where $f$, $g$, and $h$ are proper lower semi-continuous convex functions, $f$ is differentiable with a Lipschitz continuous gradient, and $\vA$ is a bounded linear operator. 
The proposed algorithm has some famous primal-dual algorithms for minimizing the sum of two functions as special cases. E.g., it reduces to the Chambolle-Pock algorithm when $f=0$ and the proximal alternating predictor-corrector when $g=0$. 
For the general convex case, we prove the convergence of this new algorithm in terms of the distance to a fixed point by showing that the iteration is a nonexpansive operator. 
In addition, we prove the $O(1/k)$ ergodic convergence rate in the primal-dual gap. 
With additional assumptions, we derive the linear convergence rate in terms of the distance to the fixed point. 
Comparing to other primal-dual algorithms for solving the same problem, this algorithm extends the range of acceptable parameters to ensure its convergence and has a smaller per-iteration cost.
The numerical experiments show the efficiency of this algorithm.

\keywords{fixed-point iteration \and nonexpansive operator \and Chambolle-Pock \and primal-dual \and three-operator splitting}
\end{abstract}

\section{Introduction}

This paper focuses on minimizing the sum of three proper lower semi-continuous convex functions in the form of
\begin{align}\label{for:main_problem}
\vx^* \in \argmin_{\vx\in\cX} f(\vx) + g(\vx) + h\square l(\vA\vx),
\end{align}
where $\cX$ and $\cS$ are two real Hilbert spaces, $h\square l:\cS\mapsto (-\infty,+\infty]$ is the infimal convolution defined as $h\square l(\vs)=\inf_{\vt}h(\vt)+l(\vs-\vt)$, 
$\vA:\cX\mapsto \cS$ is a bounded linear operator.
$f:\cX\mapsto \vR$ and the conjugate function\footnote{The conjugate function $l^*$ of $l$ is defined as $l^*(\vs)=\sup_{\vt}\langle\vs,\vt\rangle-l(\vt)$. 
When $l(\vs)=\iota_{\{\vzero\}}(\vs)$, we have $l^*(\vs)=0$. The conjugate function of $h\square l$ is $(h\square l)^*=h^*+l^*$.} $l^*:\cS\mapsto\vR$ are differentiable with Lipschitz continuous gradients, 
 and both $g$ and $h$ are proximal, that is, the proximal mappings of $g$ and $h$ defined as
\begin{align*}
\prox_{\lambda g}(\widetilde \vx)=(\vI+\lambda\partial g)^{-1}(\widetilde\vx):=\argmin_\vx~ \lambda g(\vx)+{1\over 2}\|\vx-\widetilde\vx\|^2
\end{align*}
have analytical solutions or can be computed efficiently.
When $l(\vs)=\iota_{\{\vzero\}}(\vs)$ is the indicator function that returns zero when $\vs=\vzero$ and $+\infty$ otherwise,
the infimal convolution $h\square l=h$, and the problem~\eqref{for:main_problem} becomes
\begin{align*}
\vx^* \in \argmin_{\vx\in\cX} f(\vx) + g(\vx) + h(\vA\vx).
\end{align*}
A wide range of problems in image and signal processing, statistics and machine learning can be formulated into this form. 
Here, we give some examples.

{\bf Elastic net regularization~\cite{zou_regularization_2005}:} The elastic net combines the $\ell_1$ and $\ell_2$ penalties to overcome the limitations of both penalties.
The optimization problem is 
\begin{align*}
\textstyle \vx^*\in\argmin\limits_{\vx\in\vR^p}~ \mu_2\|\vx\|_2^2 + \mu_1\|\vx\|_1 + l(\vA\vx,\vb),
\end{align*}
where $\vA\in \vR^{n\times p}$, $\vb\in\vR^n$, and $l$ is the loss function, which may be nondifferentiable.
The $\ell_2$ regularization term $\mu_2\|\vx\|_2^2$ is differentiable and has a Lipschitz continuous gradient.

{\bf Fused lasso~\cite{tibshirani2005sparsity}:} The fused lasso was proposed for group variable selection. 
Except the $\ell_1$ penalty, it includes a new penalty term for large changes with respect to the temporal or spatial structure such that the coefficients vary in a smooth fashion. 
The problem for fused lasso with the least squares loss is 
\begin{align}
 \vx^*\in\argmin\limits_{\vx\in\vR^p} {1\over 2}\|\vA\vx-\vb\|_2^2  + \mu_1\|\vx\|_1 + \mu_2\|\vD\vx\|_1,\label{eqn:fusedlasso}
\end{align}
where $\vA\in\vR^{n\times p}$, $\vb\in\vR^n$, and 
\begin{align*}
\vD=\left[\begin{array}{ccccc}-1&1& & & \\ &-1&1& & \\ & & \dots & \dots & \\& & &-1&1\end{array}\right]\in\vR^{(p-1)\times p}.
\end{align*}

{\bf Image restoration with two regularizations:} Many image processing problems have two or more regularizations. 
For instance, in computed tomography reconstruction, nonnegative constraint and total variation regularization are applied. 
The optimization problem can be formulated as 
\begin{align*}
 \vx^*\in\argmin\limits_{\vx\in\vR^n} {1\over 2}\|\vA\vx-\vb\|_2^2  + \iota_{C}(\vx)  +  \mu\|\vD\vx\|_1,
\end{align*}
where $\vx$ is the image to be reconstructed, $\vA\in \vR^{m\times n}$ is the linear forward projection operator that maps the image to the sinogram data, $\vb\in\vR^m$ is the measured sinogram data with noise, $\iota_{C}$ is the indicator function that returns zero if $\vx\in C$ (here, $C$ is the set of nonnegative vectors in $\vR^n$) and $+\infty$ otherwise, $\vD$ is a discrete gradient operator, and the last term is the (an)isotropic total variation regularization.

Before introducing algorithms for solving~\eqref{for:main_problem}, we discuss special cases of~\eqref{for:main_problem} with only two functions.
When either $f$ or $g$ is missing, the problem~\eqref{for:main_problem} reduces to the sum of two functions, and many splitting and proximal algorithms have been proposed and studied in the literature. 
Two famous groups of methods are Alternating Direction of Multiplier Method (ADMM)~\cite{gabay1976dual,glowinski1975approximation} and primal-dual algorithms~\cite{PlayDual}. 
ADMM applied on a convex optimization problem was shown to be equivalent to Douglas-Rachford Splitting (DRS) applied on the dual problem by~\cite{gabay1983chapter}, and~\cite{yan2014self} showed recently that it is also equivalent to DRS applied on the same primal problem. 
In fact, there are many different ways to reformulate a problem into a separable convex problem with linear constraints such that ADMM can be applied, and among these ways, some are equivalent. 
However, there will be always a subproblem involving $\vA$, and it may not be solved analytically depending on the properties of $\vA$ and the way ADMM is applied. 
On the other hand, primal-dual algorithms only need the operator $\vA$ and its adjoint operator $\vA^\top$\footnote{The adjoint operator $\vA^\top$ is defined by $\langle \vs,\vA\vx\rangle_\cS = \langle \vA^\top\vs,\vx\rangle_\cX$}. 
Thus, they have been applied in a lot of applications because the subproblems would be easy to solve if the proximal mappings for both $g$ and $h$ can be computed easily. 

The primal-dual algorithms for two and three functions are reviewed by~\cite{PlayDual} and specially for image processing problems by~\cite{chambolle2016introduction}. 
When the differentiable function $f$ is missing, the primal-dual algorithm is Chambolle-Pock, (see e.g.,~\cite{pock2009algorithm,esser2010general,chambolle2011first}), while the primal-dual algorithm with $g$ missing (Primal-Dual Fixed-Point algorithm based on the Proximity Operator (PDFP$^2$O) or Proximal Alternating Predictor-Corrector (PAPC)) is proposed in~\cite{loris2011generalization,chen2013primal,Drori2015209}. 
In order to solve problem~\eqref{for:main_problem} with three functions, we can reformulate the problem and apply the primal-dual algorithms for two functions. 
E.g., we can let $\bar h([\vI;~\vA]\vx) = g(\vx)+h(\vA\vx)$ and apply PAPC or let $\bar g(\vx)=f(\vx)+g(\vx)$ and apply Chambolle-Pock. 
However, the first approach introduces more dual variables and may need more iterations to obtain the same accuracy. For the second approach, the proximal mapping of $\bar g$ may not be easy to compute, and the differentiability of $f$ is not utilized.

When all three terms are present, the problem was firstly considered in~\cite{combettes2012primal} using a forward-backward-forward scheme, where two gradient evaluations and four linear operators are needed in one iteration. 
Later,~\cite{condat2013primal} and~\cite{vu2013splitting} proposed a primal-dual algorithm for~\eqref{for:main_problem}, which we call Condat-Vu. 
It is a generalization of Chambolle-Pock by involving the differentiable function $f$ with a more restrictive range for acceptable parameters than Chambolle-Pock. 
As noted in~\cite{PlayDual}, there was no generalization of PAPC for three functions at that time. 
Later, a generalization of PAPC--a Primal-Dual Fixed-Point algorithm (PDFP)--is proposed by~\cite{chen2016primal}, in which two proximal mappings of $g$ are needed in each iteration. 
But, this algorithm has a larger range of acceptable parameters than Condat-Vu.
Recently the Asymmetric Forward-Backward-Adjoint splitting (AFBA) is proposed by~\cite{latafat2016asymmetric}, and it includes Condat-Vu and PAPC as two special cases. 
This generalization has a conservative stepsize, which is the same as Condat-Vu, but only one proximal mapping of $g$ is needed in each iteration. 

In this paper, we will give a new generalization of both Chambolle-Pock and PAPC. 
This new algorithm employs the same regions of acceptable parameters with PDFP and the same per-iteration complexity as Condat-Vu and AFBA. 
In addition, when $\vA=\vI$, it recovers the three-operator splitting scheme developed by~\cite{davis2015three}, which we call Davis-Yin.
Note that Davis-Yin is a generalization of many existing two-operator splitting schemes such as forward-backward splitting~\cite{PASSTY1979383}, backward-forward splitting~\cite{peng2016coordinate,AMTL}, Peaceman-Rachford splitting~\cite{Douglas1956NumerSol,PRS}, and forward-Douglas-Rachford splitting~\cite{FDRS}.

The proposed algorithm has the following iteration:
\begin{subequations}
\begin{align*}
\vx^{k}		& =\prox_{\gamma g} (\vz^k);\\
\vs^{k+1} & =\prox_{\delta h^*} \big((\vI-\gamma\delta \vA\vA^\top)\vs^k-{\delta}\nabla l^*(\vs^k) + {\delta}\vA\left(2\vx^k-\vz^k-{\gamma}\nabla f(\vx^k)\right)\big);\\
\vz^{k+1} & =\vx^k-{\gamma}\nabla f(\vx^k) -{\gamma}\vA^\top\vs^{k+1}.
\end{align*}
\end{subequations}
Here $\gamma$ and $\delta$ are the primal and dual stepsizes, respectively. 
The convergence of this algorithm requires conditions for both stepsizes, as shown in Theorem~\ref{thm:main}. 
Since this is a primal-dual algorithm for three functions ($h\square l$ is considered as one function) and it has connections to the three-operator splitting scheme in~\cite{davis2015three}, we name it as Primal-Dual Three-Operator splitting (PD3O).
For simplicity of notation in the following sections, we may use $(\vz,\vs)$ and $(\vz^+,\vs^+)$ for the current and next iterations. 
Therefore, one iteration of PD3O from $(\vz,\vs)$ to $(\vz^+,\vs^+)$ is
\begin{subequations}\label{for:PD3O}
\begin{align}
\vx 			& = \prox_{\gamma g} (\vz);\label{for:PD3O_iteration_a}\\
\vs^+     & =\prox_{\delta h^*} \big((\vI-\gamma\delta \vA\vA^\top)\vs-{\delta}\nabla l^*(\vs) + {\delta}\vA\left(2\vx-\vz-{\gamma}\nabla f(\vx)\right)\big); \label{for:PD3O_iteration_b}\\
\vz^+     & =\vx-{\gamma}\nabla f(\vx) -{\gamma}\vA^\top\vs^+.\label{for:PD3O_iteration_c}
\end{align}
\end{subequations}
We define one PD3O iteration as an operator $\vT_{\textnormal{PD3O}}$ and $(\vz^+,\vs^+) = \vT_{\textnormal{PD3O}}(\vz,\vs)$.

The contributions of this paper can be summarized as follows:
\begin{itemize}
\item We proposed a new primal-dual algorithm for solving an optimization problem with three functions $f(\vx)+g(\vx)+h\square l(\vA\vx)$ that recovers Chambolle-Pock and PAPC for two functions with either $f$ or $g$ missing. Comparing to three existing primal-dual algorithms for solving the same problem: Condat-Vu (\cite{condat2013primal,vu2013splitting}), AFBA (\cite{latafat2016asymmetric}), and PDFP (\cite{chen2016primal}), this new algorithm combines the advantages of all three methods: the low per-iteration complexity of Condat-Vu and AFBA and the large range of acceptable parameters for convergence of PDFP. The numerical experiments show the advantage of the proposed algorithm over these existing algorithms. 

\item We prove the convergence of the algorithm by showing that the iteration is an $\alpha$-averaged operator. This result is stronger than the result for PAPC in~\cite{chen2013primal}, where the iteration is shown to be nonexpansive only. Also, we show that Chambolle-Pock is equivalent to DRS under a different metric from the previous result that it is equivalent to a proximal point algorithm applied on the Karush-Kuhn-Tucker (KKT) conditions.
\item We show the ergodic convergence rate for the primal-dual gap. The convergent sequence is different from that in~\cite{Drori2015209} when it reduces to PAPC.

\item This new algorithm also recovers Davis-Yin for minimizing the sum of three functions and thus many splitting schemes involving two operators such as forward-backward splitting, backward-forward splitting, Peaceman-Rachford splitting, and forward-Douglas-Rachford splitting.

\item With additional assumptions on the functions, we show the linear convergence rate of PD3O.
\end{itemize}

The rest of the paper is organized as follows. 
We compare PD3O with several existing primal-dual algorithms and Davis-Yin in Section~\ref{sec:compare}. 
Then we show the convergence of PD3O for the general case and its linear convergence rate for special cases in Section~\ref{sec:general}.
The numerical experiments in Section~\ref{sec:numerical} show the effectiveness of the proposed algorithm by comparing with other existing algorithms, and finally, Section~\ref{sec:conclusion} concludes the paper with future directions.

\section{Connections to existing algorithms}\label{sec:compare}
In this section, we compare our proposed algorithm with several existing algorithms. 
In particular, we show that our proposed algorithm recovers Chambolle-Pock~\cite{chambolle2011first}, PAPC~\cite{loris2011generalization,chen2013primal,Drori2015209}, and Davis-Yin~\cite{davis2015three}.
In addition, we compare our algorithm with PDFP~\cite{chen2016primal}, Condat-Vu~\cite{condat2013primal,vu2013splitting}, and AFBA~\cite{latafat2016asymmetric}.

Before showing the connections, we reformulate our algorithm by changing the update order of the variables and introducing $\bar\vx$ to replace $\vz$ (i.e., $\bar\vx = 2\vx-\vz-\gamma\nabla f(\vx)-\gamma \vA^\top\vs$). The reformulated algorithm is 
\begin{subequations}\label{for:PD3Ov2}
\begin{align}
\vs^+     & =\prox_{\delta h^*}\left(\vs-\delta \nabla l^*(\vs) + \delta\vA\bar\vx\right);\\
\vx^+ 		& =\prox_{\gamma g} (\vx -\gamma\nabla f(\vx)-{\gamma}\vA^\top\vs^+);\\
\bar\vx^+ & =2\vx^+-\vx +\gamma \nabla f(\vx)-\gamma \nabla f(\vx^+).
\end{align}
\end{subequations}
Since the infimal convolution $h\square l$ is only considered by Vu in~\cite{vu2013splitting}. 
For simplicity, we let $l=\iota_{\{\vzero\}}$, thus $l^*=0$, for the rest of this section.

\subsection{Three special cases}
In this subsection, we show three special cases of our new algorithm: PAPC, Chambolle-Pock, and Davis-Yin. 

{\bf PAPC}: 
When $g=0$, i.e., the function $g$ is missing, we have $\vx=\vz$, and the iteration~\eqref{for:PD3O} reduces to
\begin{subequations}\label{for:PDFP2O}
\begin{align}
\vs^+     & =\prox_{\delta h^*} ((\vI-\gamma\delta \vA\vA^\top)\vs + \delta\vA\left(\vx-{\gamma}\nabla f(\vx)\right));\\
\vx^+     & =\vx-{\gamma}\nabla f(\vx) -{\gamma}\vA^\top\vs^+,
\end{align}
\end{subequations}
which is PAPC in~\cite{loris2011generalization,chen2013primal,Drori2015209}. The PAPC iteration is shown to be nonexpansive in~\cite{chen2013primal}, while we will show a stronger result that the PAPC iteration is $\alpha$-averaged with certain $\alpha\in(0,1)$ in Corollary~\ref{cor:PDFP2O}. 
In addition, we prove the convergence rate of the primal-dual gap using a different sequence from~\cite{Drori2015209}. See Section~\ref{sec:primal_dual}.

{\bf Chambolle-Pock}: 
Let $f=0$, i.e., the function $f$ is missing, then we have, from~\eqref{for:PD3Ov2},
\begin{subequations}
\begin{align}
\vs^+     & =\prox_{\delta h^*} (\vs + \delta\vA\bar\vx);\\
\vx^+ 		& =\prox_{\gamma g} (\vx -{\gamma}\vA^\top\vs^+);\\
\bar\vx^+ & =2\vx^+-\vx,
\end{align}
\end{subequations}
which is Chambolle-Pock in~\cite{chambolle2011first}. We will show that Chambolle-Pock is equivalent to a Douglas-Rachford splitting under a metric in Corollary~\ref{cor:CP}.

{\bf Davis-Yin}: 
Let $\vA=\vI$ and $\gamma\delta=1$, then we have, from~\eqref{for:PD3O},
\begin{subequations}\label{for:3O}
\begin{align}
\vx        &=\prox_{\gamma g} (\vz);\label{for:3O_iteration_a}\\
\vs^+      &= \prox_{{\delta}h^*}\left(\delta(2\vx-\vz-{\gamma}\nabla f(\vx))\right)\nonumber \\
           &={\delta}(\vI-\prox_{{\gamma} h}) \left(2\vx-\vz-{\gamma}\nabla f(\vx)\right);\label{for:3O_iteration_b} \\
\vz^+      &=\vx-{\gamma}\nabla f(\vx) -{\gamma}\vs^+.\label{for:3O_iteration_c}
\end{align}
\end{subequations}
Here we used the Moreau decomposition in~\eqref{for:3O_iteration_b}. Combining~\eqref{for:3O} together, we have 
\begin{align*}
\vz^+      =\vz+ \prox_{{\gamma} h}\left(2\prox_{\gamma g} (\vz)-\vz-{\gamma}\nabla f(\prox_{\gamma g} (\vz))\right) -\prox_{\gamma g} (\vz),
\end{align*}
which is Davis-Yin in~\cite{davis2015three}.

\subsection{Comparison with three primal-dual algorithms for three functions}
In this subsection, we compare our algorithm with three primal-dual algorithms for solving the same problem~\eqref{for:main_problem}.

{\bf PDFP}: 
The PDFP algorithm~\cite{chen2016primal} is developed as a generalization of PAPC. When $g=\iota_C$ ($C$ is a convex set in $\cX$), PDFP reduces to the Preconditioned Alternating Projection Algorithm (PAPA) proposed in~\cite{krol2012preconditioned}. 
The PDFP iteration can be expressed as follows:

\begin{subequations}\label{for:PDFP}
\begin{align}
\vs^+     &\textstyle  =\prox_{\delta h^*} \left(\vs + {\delta}\vA\bar\vx\right); \\
\vx^+     &\textstyle  =\prox_{\gamma g} (\vx-\gamma \nabla f(\vx)-\gamma\vA^\top\vs^+);\\
\bar\vx^+ &\textstyle  =\prox_{\gamma g} (\vx^+-\gamma \nabla f(\vx^+)-\gamma\vA^\top\vs^+).
\end{align}
\end{subequations}
Note that two proximal mappings of $g$ are needed in each iteration, while other algorithms only need one.

{\bf Condat-Vu}: The Condat-Vu algorithm~\cite{condat2013primal,vu2013splitting} is a generalization of the Chambolle-Pock algorithm for problem~\eqref{for:main_problem}. The iteration is
\begin{subequations}\label{for:CV}
\begin{align}
\vs^+     & =\prox_{\delta h^*} (\vs + \delta\vA\bar\vx);\\
\vx^+ 		& =\prox_{\gamma g} (\vx -\gamma\nabla f(\vx)-{\gamma}\vA^\top\vs^+);\\
\bar\vx^+ & =2\vx^+-\vx.
\end{align}
\end{subequations}
The difference between our algorithm and Condat-Vu is in the updating of $\bar\vx$. Because of the difference, our algorithm will be shown to have more freedom than Condat-Vu in choosing acceptable parameters.
Note though~\cite{vu2013splitting} considers a general form with infimal convolution, its condition for the parameters, when reduced to the case without the infimal convolution, is worse than that in~\cite{condat2013primal}. 
The correct general condition has been given in~\cite[Theorem 5]{lorenz2015inertial}.

{\bf AFBA}: AFBA is a very general operator splitting scheme, and one of its special cases~\cite[Algorithm 5]{latafat2016asymmetric} can be used to solve the problem~\eqref{for:main_problem}. 
The corresponding algorithm is
\begin{subequations}\label{for:AFBA}
\begin{align}
\vs^+     & =\prox_{\delta h^*} (\vs + \delta\vA\bar\vx); \\
\vx^+     & =\bar\vx-\gamma\vA^\top(\vs^+-\vs);\\
\bar\vx^+       & =\prox_{\gamma g} (\vx^+-\gamma \nabla f(\vx^+)-\gamma\vA^\top\vs^+).
\end{align}
\end{subequations}
The difference between AFBA and PDFP is in the update of $\vx^+$. Though AFBA needs only one proximal mapping of $g$ in each iteration, it has a more conservative range for the two parameters than PDFP and Condat-Vu.

The parameters for the four algorithms solving~\eqref{for:main_problem} and their relations to the aforementioned primal-dual algorithms for the sum of two functions are given in Table~\ref{tab:compare}. 
Here $\|\vA\vA^\top\|:=\max\limits_{\|\vs\|=1}{\|\vA\vA^\top\vs\|}$.
\begin{table}[!h]
\begin{center}
\begin{tabular}{|c|c|c|c|}\hline
									& $f\neq0,~g\neq0$	&  $f=0$  & $g=0$\\\hline
PDFP 							& $\gamma\delta\|\vA\vA^\top\|<1$; $\gamma/(2\beta)<1$	&  &PAPC \\\hline
Condat-Vu 				& $\gamma\delta \|\vA\vA^\top\|+\gamma/(2\beta)\leq 1$	&  Chambolle-Pock & \\\hline
AFBA 				& $\gamma\delta \|\vA\vA^\top\|/2+\sqrt{\gamma\delta \|\vA\vA^\top\|}/2+\gamma/(2\beta)\leq 1$	&   & PAPC\\\hline
PD3O	  & $\gamma\delta\|\vA\vA^\top\|<1$; $\gamma/(2\beta)<1$ &  Chambolle-Pock  &PAPC\\\hline
\end{tabular}
\end{center}
\caption{The comparison of convergence conditions for PDFP, Condat-Vu, AFBA, and PD3O and their reduced primal-dual algorithms.}\label{tab:compare}
\end{table}

Comparing the iterations of PDFP, Condat-Vu, and PD3O in~\eqref{for:PDFP},~\eqref{for:CV}, and~\eqref{for:PD3Ov2}, respectively, we notice that the difference is in the third step for updating $\bar\vx$. 
The third steps for these three algorithms are summarized below:
\begin{center}
\begin{tabular}{|l|l|}\hline
PDFP & $\bar\vx^+  =\prox_{\gamma g} (\vx^+-\gamma \nabla f(\vx^+)-\gamma\vA^\top\vs^+)$\\\hline
Condat-Vu & $\bar\vx^+  =2\vx^+-\vx$ \\\hline
PD3O & $\bar\vx^+ = 2\vx^+-\vx +\gamma \nabla f(\vx)-\gamma \nabla f(\vx^+)$\\\hline
\end{tabular}
\end{center}
Though there are two more terms ($\nabla f(\vx)$ and $\nabla f(\vx^+)$) in PD3O than Condat-Vu, $\nabla f(\vx)$ has been computed in the previous step, and $\nabla f(\vx^+)$ will be used in the next iteration. 
Thus, except that $\nabla f(\vx)$ has to be stored, there is no additional cost in PD3O comparing to Condat-Vu. 
However, for PDFP, the proximal operator $\prox_{\gamma g}$ is applied twice on different values in each iteration, and it will not be used in the next iteration. 
Therefore,  the per-iteration cost is more than the other three algorithms. 
When the proximal mapping is simple, the additional cost can be relatively small compared to other operations in one iteration. 

\section{The proposed primal-dual algorithm}\label{sec:general}
\subsection{Notation and preliminaries}
Let $\vI$ be the identity operator defined on a real Hilbert space. For simplicity, we do not specify the space on which it is defined when it is clear from the context.
Let $\vM={\gamma\over \delta}(\vI - \gamma\delta\vA\vA^\top)$ and $\langle \vs_1,\vs_2\rangle_\vM:=\langle \vs_1, \vM\vs_2\rangle$ for $\vs_1,~\vs_2\in\cS$. 
When $\gamma\delta$ is small enough such that $\vM$ is positive semidefinite, we define $\|\vs\|_\vM=\sqrt{\langle \vs,\vs\rangle_\vM}$ for any $\vs\in\cS$ and $\|(\vz,\vs)\|_{\vI,\vM}=\sqrt{\|\vz\|^2+\|\vs\|^2_\vM}$ for any $(\vz,\vs)\in \cX\times \cS$. 
Specially, when $\vM$ is positive definite, $\|\cdot\|_\vM$ and $\|(\cdot,\cdot)\|_{\vI,\vM}$ are two norms defined on $\cS$ and $\cX\times\cS$, respectively. 
From~\eqref{for:PD3O}, we define $\vq_{h^*}(\vs^+)\in \partial h^*(\vs^+)$ and $\vq_g(\vx)\in\partial g(\vx)$ as:
\begin{align*}
\vq_{h^*}(\vs^+):=& \textstyle{\gamma}^{-1}\vM\vs-\nabla l^*(\vs)+\vA(2\vx-\vz-{\gamma}\nabla f(\vx))-{\delta}^{-1}\vs^+ \in \partial h^*(\vs^+),\\
\vq_g(\vx):      =& \textstyle{\gamma}^{-1}(\vz-\vx)\in\partial g(\vx).
\end{align*}
An operator $\vT$ is nonexpansive if $\|\vT\vx_1 -\vT\vx_2\|\leq \|\vx_1-\vx_2\|$ for any $\vx_1$ and $\vx_2$. 
An operator $\vT$ is $\alpha$-averaged for $\alpha\in(0,1]$ if $\|\vT\vx_1 -\vT\vx_2\|^2\leq \|\vx_1-\vx_2\|^2-((1-\alpha)/\alpha)\|(\vI-\vT)\vx_1-(\vI-\vT)\vx_2\|^2$; firmly nonexpansive operators are $1/2$-averaged, and merely nonexpansive operators are $1$-averaged.

\begin{assumption}\label{assum:1}
We assume that functions $f$, $g$, $h$, and $l$ are proper lower semi-continuous convex and there exists $\beta>0$ such that 
\begin{align}
\langle \vx_1-\vx_2,\nabla f(\vx_1)-\nabla f(\vx_2)\rangle\geq & \beta\|\nabla f(\vx_1)-\nabla f(\vx_2)\|^2, \label{eq:fcoco}\\
\langle \vs_1-\vs_2,\nabla l^*(\vs_1)-\nabla l^*(\vs_2)\rangle\geq &\beta\|\nabla l^*(\vs_1)-\nabla l^*(\vs_2)\|^2_{\vM^{-1}}\label{eq:lcoco}
\end{align}
are satisfied for any $\vx_1,~\vx_2\in \cX$ and $\vs_1,~\vs_2\in\cS$. 
We assume that $\vM$ is positive definite if $\nabla l^*$ is not a constant.
When $\nabla l^*$ is a constant, the left hand side of equation~\eqref{eq:lcoco} is zero, and we say that the equation is satisfied for any positive $\beta$ and positive semidefinite $\vM$. 
Therefore, when $\nabla l^*$ is a constant, we only assume that $\vM$ is positive semidefinite.
\end{assumption}
Equation~\eqref{eq:fcoco} is satisfied if and only if $f$ has a $1/\beta$ Lipschitz continuous gradient when $f$ is proper lower semi-continuous convex (see \cite[Theorem 18.15]{bauschke2011convex}). 
We assume~\eqref{eq:lcoco} with $\vM$ for the simplicity of the results. 
The results are valid but complicated when $\vM$ is replaced by $\vI$ in~\eqref{eq:lcoco}.
When $\nabla l^*$ is Lipschitz continuous, we can always find $\beta$ such that equation~\eqref{eq:lcoco} is satisfied if $\vM$ is positive definite.  
When $\det(\vM)=0$, we consider the case with $\nabla l^*$ being a constant only.

\begin{lemma}\label{lemma4gap}
If Assumption~\ref{assum:1} is satisfied, we have
\begin{align}
f(\vx_2)-f(\vx_1) \leq &\langle \nabla f(\vx_2),\vx_2-\vx_1\rangle  -{\beta\over2} \|\nabla f(\vx_2)-\nabla f(\vx_1)\|^2, \label{lemma1a}\\
l^*(\vs_1)-l^*(\vs_2) \leq & \langle \nabla l^*(\vs_2),\vs_1-\vs_2\rangle+{(2\beta)^{-1}}\|\vs_1-\vs_2\|_\vM^2. \label{lemma1b}
\end{align}
\end{lemma}
\begin{proof}
Given $\vx_2$, the function $f_{\vx_2}(\vx):=f(\vx)-\nabla f(\vx_2)^\top\vx$ is convex and has a Lipschitz continuous gradient with constant $1/\beta$. 
In addition, $\vx_2$ is a minimum of $f_{\vx_2}(\vx)$. Therefore, we have 
\begin{align}
f_{\vx_2}(\vx_2)\leq & \inf_{\vx\in\cX} f(\vx_1)-\nabla f(\vx_2)^\top\vx_1 +\langle \nabla f(\vx_1)-\nabla f(\vx_2),\vx-\vx_1\rangle + {1\over2\beta}\|\vx-\vx_1\|^2 \nonumber\\
= & f(\vx_1)-\nabla f(\vx_2)^\top\vx_1 -{\beta\over2}\|\nabla f(\vx_2)-\nabla f(\vx_1)\|^2,
\end{align}
which gives
\begin{align}
f(\vx_2)- f(\vx_1)\leq \langle\nabla f(\vx_2),\vx_2-\vx_1\rangle-{\beta\over2}\|\nabla f(\vx_2)-\nabla f(\vx_1)\|^2.
\end{align}
When $\vM$ is positive definite, the inequality~\eqref{eq:lcoco} implies the Lipschitz continuity of $ \vM^{-1/2} \nabla l^*(\vM^{-1/2}\cdot)$ with constant $1/\beta$.
Then the function $(2\beta)^{-1}\hat\vs^\top\hat\vs-l^*(\vM^{-1/2}\hat\vs)$ is convex. 
Therefore we have 
\begin{align*}
{(2\beta)^{-1}}\hat\vs_1^\top\hat\vs_1-l^*(\vM^{-1/2}\hat\vs_1) \geq &{(2\beta)^{-1}}\hat\vs_2^\top\hat\vs_2-l^*(\vM^{-1/2}\hat\vs_2) \\
&+\left\langle {\beta^{-1}}\hat\vs_2-\vM^{-1/2}\nabla l^*(\vM^{-1/2}\hat\vs_2),\hat\vs_1-\hat\vs_2 \right\rangle
\end{align*}
Let $\vs_1=\vM^{-1/2}\hat\vs_1$ and $\vs_2=\vM^{-1/2}\hat\vs_2$, then we have
\begin{align*}
l^*(\vs_1) -l^*(\vs_2) \leq \langle\nabla l^*(\vs_2),\vs_1-\vs_2 \rangle+{(2\beta)^{-1}}\|\vs_1-\vs_2\|_\vM^2.
\end{align*}
When $\nabla l^*$ is a constant,~\eqref{lemma1b} is satisfied for any positive semidefinite $\vM$.
\qed
\end{proof}

\begin{lemma}[fundamental equality]\label{lemma:fundamental} 
Consider the PD3O iteration~\eqref{for:PD3O} with two inputs $(\vz_1,\vs_1)$ and $(\vz_2,\vs_2)$. 
Let $(\vz_1^+,\vs_1^+)=\PD(\vz_1,\vs_1)$ and $(\vz_2^+,\vs_2^+)=\PD(\vz_2,\vs_2)$. 
We also define 
\begin{align*}
\vq_{h^*}(\vs_1^+):=& \textstyle {\gamma}^{-1}\vM\vs_1-\nabla l^*(\vs_1)+\vA(2\vx_1-\vz_1-{\gamma}\nabla f(\vx_1))-{\delta}^{-1}\vs_1^+ \in \partial h^*(\vs_1^+),\\
\vq_g(\vx_1):=  & \textstyle {\gamma}^{-1}(\vz_1-\vx_1)\in\partial g(\vx_1),\\
\vq_{h^*}(\vs_2^+):=& \textstyle {\gamma}^{-1}\vM\vs_2-\nabla l^*(\vs_2)+\vA(2\vx_2-\vz_2-{\gamma}\nabla f(\vx_2))-{\delta}^{-1}\vs_2^+ \in \partial h^*(\vs_2^+),\\
\vq_g(\vx_2):=  & \textstyle {\gamma}^{-1}(\vz_2-\vx_2)\in\partial g(\vx_2). 
\end{align*}
Then, we have 
\begin{align}
			& 2\gamma\langle \vs_1^+-\vs_2^+,\vq_{h^*}(\vs_1^+)-\vq_{h^*}(\vs_2^+)+\nabla l^*(\vs_1)-\nabla l^*(\vs_2)\rangle \nonumber\\
			& +2\gamma\langle \vx_1-\vx_2,\vq_g(\vx_1)-\vq_g(\vx_2)+\nabla f(\vx_1)-\nabla f(\vx_2)\rangle\nonumber\\
	=   & \|(\vz_1,\vs_1)-(\vz_2,\vs_2)\|_{\vI,\vM}^2-\|(\vz_1^+,\vs_1^+)-(\vz_2^+,\vs_2^+)\|_{\vI,\vM}^2 -\|\vz_1-\vz_1^+-\vz_2+\vz_2^+\|^2 \nonumber\\
	    & -\|\vs_1-\vs_1^+-\vs_2+\vs_2^+\|_\vM^2 +2\gamma\langle \nabla f(\vx_1)-\nabla f(\vx_2), \vz_1-\vz_1^+-\vz_2+\vz_2^+\rangle. \label{eqn:fundamental_1}
\end{align}
\end{lemma}

\begin{proof}
We consider the two terms on the left hand side of~\eqref{eqn:fundamental_1} separately. 
For the first term, we have
\begin{align}
      & 2\gamma\langle \vs_1^+-\vs_2^+,\vq_{h^*}(\vs_1^+)-\vq_{h^*}(\vs_2^+) +\nabla l^*(\vs_1)-\nabla l^*(\vs_2) \rangle\nonumber\\
	=		&\textstyle  2\gamma\langle \vs_1^+-\vs_2^+,{\gamma}^{-1}\vM\vs_1-\nabla l^*(\vs_1)+\vA(\vz_1^+-\vz_1+\vx_1+\gamma\vA^\top\vs_1^+)-{\delta}^{-1}\vs_1^+\rangle\nonumber\\
	 		&\textstyle -2\gamma\langle \vs_1^+-\vs_2^+,{\gamma}^{-1}\vM\vs_2-\nabla l^*(\vs_2)+\vA(\vz_2^+-\vz_2+\vx_2+\gamma\vA^\top\vs_2^+)-{\delta}^{-1}\vs_2^+\rangle\nonumber\\
			& +2\gamma\langle \vs_1^+-\vs_2^+,\nabla l^*(\vs_1)-\nabla l^*(\vs_2) \rangle \nonumber\\
	=		& 2\langle\vs_1^+-\vs_2^+,\vs_1-\vs_1^+-\vs_2+\vs_2^+\rangle_\vM\nonumber\\
	    & + 2\gamma\langle \vA^\top(\vs_1^+-\vs_2^+),\vz_1^+-\vz_1+\vx_1-\vz_2^++\vz_2-\vx_2\rangle. \label{eq:equal_h}
\end{align}
The updates of $\vz_1^+$ and $\vz_2^+$ in~\eqref{for:PD3O_iteration_c} show
\begin{align}
      & 2\gamma\langle \vx_1-\vx_2,\vq_g(\vx_1)-\vq_g(\vx_2)+\nabla f(\vx_1)-\nabla f(\vx_2)\rangle\nonumber\\
	=		& 2\gamma\langle \vx_1-\vx_2,{\gamma}^{-1}(\vz_1-\vx_1)-{\gamma}^{-1}(\vz_2-\vx_2)+\nabla f(\vx_1)-\nabla f(\vx_2)\rangle\nonumber\\
	=		& 2\langle \vx_1-\vx_2,\vz_1-\vx_1+\gamma\nabla f(\vx_1)-\vz_2+\vx_2-\gamma\nabla f(\vx_2)\rangle\nonumber\\
	=		& 2\langle \vx_1-\vx_2,\vz_1-\vz_1^+-\gamma\vA^\top\vs_1^+ -\vz_2+\vz_2^++\gamma \vA^\top\vs_2^+\rangle. \label{eq:equal_fg}
\end{align}
Combining both~\eqref{eq:equal_h} and~\eqref{eq:equal_fg}, we have 
\begin{align}
      & 2\gamma\langle \vs_1^+-\vs_2^+,\vq_{h^*}(\vs_1^+)-\vq_{h^*}(\vs_2^+)+\nabla l^*(\vs_1)-\nabla l^*(\vs_2)\rangle \nonumber\\
			& +2\gamma\langle \vx_1-\vx_2,\vq_g(\vx_1)-\vq_g(\vx_2)+\nabla f(\vx_1)-\nabla f(\vx_2)\rangle \nonumber\\
	=   & 2\langle\vs_1^+-\vs_2^+,\vs_1-\vs_1^+-\vs_2+\vs_2^+\rangle_\vM + 2\gamma\langle \vA^\top(\vs_1^+-\vs_2^+),\vz_1^+-\vz_1-\vz_2^++\vz_2\rangle \nonumber\\
			& + 2\gamma\langle \vA^\top(\vs_1^+-\vs_2^+),\vx_1-\vx_2\rangle - 2\langle \vx_1-\vx_2,\gamma\vA^\top\vs_1^+ -\gamma \vA^\top\vs_2^+\rangle \nonumber\\
			& +2\langle \vx_1-\vx_2,\vz_1-\vz_1^+ -\vz_2+\vz_2^+\rangle \nonumber\\
	=   & 2\langle\vs_1^+-\vs_2^+,\vs_1-\vs_1^+-\vs_2+\vs_2^+\rangle_\vM \nonumber\\
			& +2\langle \vx_1 -\gamma\vA^\top\vs_1^+-\vx_2+\gamma\vA^\top\vs_2^+,\vz_1-\vz_1^+ -\vz_2+\vz_2^+\rangle \nonumber\\
	=   & 2\langle\vs_1^+-\vs_2^+,\vs_1-\vs_1^+-\vs_2+\vs_2^+\rangle_\vM \nonumber\\
			& +2\langle \vz_1^+ +\gamma\nabla f(\vx_1)-\vz_2^+-\gamma\nabla f(\vx_2),\vz_1-\vz_1^+ -\vz_2+\vz_2^+\rangle \label{eqn:gapfunction}\\
  =   & \|(\vz_1,\vs_1)-(\vz_2,\vs_2)\|_{\vI,\vM}^2-\|(\vz_1^+,\vs_1^+)-(\vz_2^+,\vs_2^+)\|_{\vI,\vM}^2 -\|\vz_1-\vz_1^+-\vz_2+\vz_2^+\|^2 \nonumber\\
	    & -\|\vs_1-\vs_1^+-\vs_2+\vs_2^+\|_{\vM}^2 +2\gamma\langle \nabla f(\vx_1)-\nabla f(\vx_2), \vz_1-\vz_1^+-\vz_2+\vz_2^+\rangle, \nonumber
\end{align}
where the third equality comes from the updates of $\vz_1^+$ and $\vz_2^+$ in~\eqref{for:PD3O_iteration_c} and the last equality holds because of $2\langle a,b\rangle =\|a+b\|^2-\|a\|^2-\|b\|^2$. 
\qed
\end{proof}
Note that Lemma~\ref{lemma:fundamental} only requires the symmetry of $\vM$.

\subsection{Convergence analysis for the general convex case}
In this section, we show the convergence of the proposed algorithm in Theorem~\ref{thm:main}. We show firstly that the operator $\PD$ is a nonexpansive operator (Lemma~\ref{lemma:nonexpansive}) and then finding a fixed point $(\vz^*,\vs^*)$ of $\PD$ is equivalent to finding an optimal solution to~\eqref{for:main_problem} (Lemma~\ref{lemma:optimal_solution}).

\begin{lemma}\label{lemma:nonexpansive}
Let $(\vz_1^+,\vs_1^+)=\PD(\vz_1,\vs_1)$ and $(\vz_2^+,\vs_2^+)=\PD(\vz_2,\vs_2)$. 
Under Assumption~\ref{assum:1},% and $\vM$ being positive semidefinite ($\nabla l^*$ is a constant if $\det (\vM)=0$), 
we have
\begin{align}
       & \|(\vz_1^+,\vs_1^+)-(\vz_2^+,\vs_2^+)\|_{\vI,\vM}^2 -\|(\vz_1,\vs_1)-(\vz_2,\vs_2)\|_{\vI,\vM}^2 \nonumber\\
\leq 	 & -{2\beta-\gamma\over2\beta}\left(\|(\vz_1^+,\vs_1^+)-(\vz_1,\vs_1)-(\vz_2^+,\vs_2^+)+(\vz_2,\vs_2)\|_{\vI,\vM}^2\right).\label{eq:average}
\end{align}
Furthermore, when $\vM$ is positive definite, the operator $\PD$ in~\eqref{for:PD3O} is nonexpansive for $(\vz,\vs)$ if $\gamma\leq 2\beta$. 
More specifically, it is $\alpha$-averaged with $\alpha={2\beta\over 4\beta-\gamma}$.
\end{lemma}

\begin{proof}
Because of the convexity of $h^*$ and $g$, we have
\begin{align*}
0\leq & \langle \vs_1^+-\vs_2^+,\vq_{h^*}(\vs_1^+)-\vq_{h^*}(\vs_2^+)\rangle
 +\langle \vx_1-\vx_2,\vq_g(\vx_1)-\vq_g(\vx_2)\rangle,
\end{align*}
where $\vq_{h^*}(\vs_1^+)$, $\vq_{h^*}(\vs_1^+)$, $\vq_g(\vx_1)$, and $\vq_g(\vx_2)$ are defined in Lemma~\ref{lemma:fundamental}.
Then, the fundamental equality~\eqref{eqn:fundamental_1} in Lemma~\ref{lemma:fundamental} gives
\begin{align}
      & \|(\vz_1^+,\vs_1^+)-(\vz_2^+,\vs_2^+)\|_{\vI,\vM}^2 - \|(\vz_1,\vs_1)-(\vz_2,\vs_2)\|_{\vI,\vM}^2 \nonumber\\
\leq	& 2\gamma\langle \nabla f(\vx_1)-\nabla f(\vx_2), \vz_1-\vz_1^+-\vz_2+\vz_2^+\rangle -2\gamma\langle \nabla f(\vx_1)-\nabla f(\vx_2), \vx_1-\vx_2\rangle\nonumber\\
      & -2\gamma\langle \vs_1^+-\vs_2^+,\nabla l^*(\vs_1)-\nabla l^*(\vs_2)\rangle \label{eq:nonexpansive_cp}\\
			& \textstyle - \|\vz_1-\vz_1^+-\vz_2+\vz_2^+\|^2 -  \|\vs_1-\vs_1^+-\vs_2+\vs_2^+\|_\vM^2.\nonumber
\end{align}
Next we derive the upper bound of the cross terms in~\eqref{eq:nonexpansive_cp} as follows:
\begin{align}
      & 2\gamma\langle \nabla f(\vx_1)-\nabla f(\vx_2), \vz_1-\vz_1^+-\vz_2+\vz_2^+\rangle -2\gamma\langle \nabla f(\vx_1)-\nabla f(\vx_2), \vx_1-\vx_2\rangle\nonumber\\
      & -2\gamma\langle \vs_1^+-\vs_2^+,\nabla l^*(\vs_1)-\nabla l^*(\vs_2)\rangle \nonumber\\
=			& 2\gamma\langle \nabla f(\vx_1)-\nabla f(\vx_2), \vz_1-\vz_1^+-\vz_2+\vz_2^+\rangle - 2\gamma\langle \nabla f(\vx_1)-\nabla f(\vx_2), \vx_1-\vx_2\rangle \nonumber\\
      & + 2\gamma\langle \vs_1-\vs_1^+-\vs_2+\vs_2^+,\nabla l^*(\vs_1)-\nabla l^*(\vs_2)\rangle -2\gamma\langle \vs_1-\vs_2,\nabla l^*(\vs_1)-\nabla l^*(\vs_2)\rangle \nonumber\\
\leq 	& 2\gamma\langle \nabla f(\vx_1)-\nabla f(\vx_2) , \vz_1-\vz_1^+-\vz_2+\vz_2^+\rangle - 2\gamma\beta \| \nabla f(\vx_1)-\nabla f(\vx_2)\|^2 \nonumber\\			
			& +2\gamma\langle \vs_1-\vs_1^+-\vs_2+\vs_2^+,\nabla l^*(\vs_1)-\nabla l^*(\vs_2)\rangle -2\gamma\beta\|\nabla l^*(\vs_1)-\nabla l^*(\vs_2)\|^2_{\vM^{-1}} \nonumber\\
\leq	& \textstyle \epsilon \|\vz_1-\vz_1^+-\vz_2+\vz_2^+\|^2+\left({\gamma^2\over \epsilon}-2\gamma\beta\right)\| \nabla f(\vx_1)-\nabla f(\vx_2)\|^2 \nonumber\\
      & \textstyle+\epsilon\|\vs_1-\vs_1^+-\vs_2+\vs_2^+\|_\vM^2+\left({\gamma^2\over \epsilon}-2\gamma\beta\right)\|\nabla l^*(\vs_1)-\nabla l^*(\vs_2)\|^2_{\vM^{-1}}. \label{lemma2_a}
\end{align}
The first inequality comes from the cocoerciveness of $\nabla f$ and $\nabla l^*$ in Assumption~\ref{assum:1}, and the second inequality comes from the Cauchy-Schwarz inequality. Therefore, combing~\eqref{eq:nonexpansive_cp} and~\eqref{lemma2_a} gives
\begin{align}
			&	\|(\vz_1^+,\vs_1^+)-(\vz_2^+,\vs_2^+)\|_{\vI,\vM}^2-\|(\vz_1,\vs_1)-(\vz_2,\vs_2)\|_{\vI,\vM}^2\nonumber\\
\leq 	& \textstyle  - (1-\epsilon)\|(\vz_1^+,\vs_1^+)-(\vz_1,\vs_1)-(\vz_2^+,\vs_2^+)+(\vz_2,\vs_2)\|_{\vI,\vM}^2 \label{eq:nonexpansive_a}\\
			& \textstyle +\left({\gamma^2\over \epsilon}-2\gamma\beta\right)(\| \nabla f(\vx_1)-\nabla f(\vx_2)\|^2+\|\nabla l^*(\vs_1)-\nabla l^*(\vs_2)\|_{\vM^{-1}}^2). \nonumber
\end{align}
In addition, letting $\epsilon=\gamma/(2\beta)$, we have that 
\begin{align}
       & \|(\vz_1^+,\vs_1^+)-(\vz_2^+,\vs_2^+)\|_{\vI,\vM}^2 -\|(\vz_1,\vs_1)-(\vz_2,\vs_2)\|_{\vI,\vM}^2 \nonumber\\
\leq 	 & -{2\beta-\gamma\over2\beta}\left(\|(\vz_1^+,\vs_1^+)-(\vz_1,\vs_1)-(\vz_2^+,\vs_2^+)+(\vz_2,\vs_2)\|_{\vI,\vM}^2\right),\label{eq:nonexpansive_b}
\end{align}
Furthermore, when $\vM$ is positive definite, $\PD$ is $\alpha-$averaged with $\alpha ={2\beta\over 4\beta-\gamma}$ under the norm $\|(\cdot,\cdot)\|_{\vI,\vM}$. \qed
\end{proof}

%Note that when $\nabla l^*$ is not a constant, the condition~\eqref{eq:lcoco} requires the parameters $\gamma$ and $\delta$.

\begin{remark}\label{remark1}
For Davis-Yin (i.e., $\vA=\vI$ and $\gamma\delta=1$), we have $\vM=\vzero$ (therefore, $\nabla l^*$ is a constant) and~\eqref{eq:nonexpansive_a} becomes 
\begin{align*}
\|\vz_1^+-\vz_2^+\|^2-\|\vz_1-\vz_2\|^2  \leq  &   - (1-\epsilon)\|\vz_1^+-\vz_1-\vz_2^++\vz_2\|^2 \\
			                  & \textstyle +\left({\gamma^2\over \epsilon}-2\gamma\beta\right)\| \nabla f(\vx_1)-\nabla f(\vx_2)\|^2,
\end{align*}
which is similar to that of Remark 3.1 in~\cite{davis2015three}. 
In fact, the nonexpansiveness of $\PD$ in Lemma~\ref{lemma:nonexpansive} can also be derived by modifying the result of the three-operator splitting under the new norm defined by $\|(\cdot,\cdot)\|_{\vI,\vM}$ (for positive definite $\vM$). 
The equivalent problem is 
\begin{align*}
\left[\begin{array}{c} \vzero\\\vzero\end{array}\right]\in&\left[\begin{array}{cc}\nabla f&0\\0&\vM^{-1}\nabla l^*\end{array}\right]\left[\begin{array}{c} \vx\\\vs\end{array}\right]
+\left[\begin{array}{cc}\partial g&0\\0&0\end{array}\right]\left[\begin{array}{c} \vx\\\vs\end{array}\right] +\left[\begin{array}{cc}\vI&0\\0&\vM^{-1}\end{array}\right]\left[\begin{array}{cc} 0&\vA^\top\\-\vA&\partial h^*\end{array}\right]\left[\begin{array}{c} \vx\\\vs\end{array}\right].
\end{align*}
In this case, Assumption~\ref{assum:1} provides
\begin{align*}
&\left\langle\left[\begin{array}{cc}\nabla f&0\\0&\vM^{-1}\nabla l^*\end{array}\right]\left[\begin{array}{c} \vz_1\\\vs_1\end{array}\right]-
\left[\begin{array}{cc}\nabla f&0\\0&\vM^{-1}\nabla l^*\end{array}\right]\left[\begin{array}{c} \vz_2\\\vs_2\end{array}\right],\left[\begin{array}{c} \vz_1-\vz_2\\\vs_1-\vs_2\end{array}\right]
\right\rangle_{\vI,\vM}\\
\geq & \beta\left\|\left[\begin{array}{cc}\nabla f&0\\0&\vM^{-1}\nabla l^*\end{array}\right]\left[\begin{array}{c} \vz_1\\\vs_1\end{array}\right]-
\left[\begin{array}{cc}\nabla f&0\\0&\vM^{-1}\nabla l^*\end{array}\right]\left[\begin{array}{c} \vz_2\\\vs_2\end{array}\right]\right\|_{\vI,\vM}^2,
\end{align*}
i.e., $\left[\begin{array}{cc}\nabla f&0\\0&\vM^{-1}\nabla l^*\end{array}\right]$ is a $\beta$-cocoercive operator under the norm $\|(\cdot,\cdot)\|_{\vI,\vM}$. 
In addition, the two operators $\left[\begin{array}{cc}\partial g&0\\0&0\end{array}\right]$ and $\left[\begin{array}{cc}\vI&0\\0&\vM^{-1}\end{array}\right]\left[\begin{array}{cc} 0&\vA^\top\\-\vA&\partial h^*\end{array}\right]$ are maximal monotone under the norm $\|(\cdot,\cdot)\|_{\vI,\vM}$. 
Then we can modify the result of Davis-Yin and show the $\alpha$-averageness of $\PD$. 
However, the primal-dual gap convergence in Section~\ref{sec:primal_dual} and linear convergence rate in Section~\ref{sec:linear} can not be obtained from~\cite{davis2015three} because $\left[\begin{array}{cc}\nabla f&0\\0&0\end{array}\right]$ is not strongly monotone under the norm $\|(\cdot,\cdot)\|_{\vI,\vM}$. 
For the completeness, we provide the proof of the $\alpha$-averageness of $\PD$ here. In fact,~\eqref{eq:nonexpansive_b} in the proof can be stronger than the $\alpha$-averageness result when $\nabla l^*=\vzero$ or $\nabla f=\vzero$. For example, when $\nabla l^*=\vzero$, we have 
\begin{align}
			&	\|(\vz_1^+,\vs_1^+)-(\vz_2^+,\vs_2^+)\|_{\vI,\vM}^2-\|(\vz_1,\vs_1)-(\vz_2,\vs_2)\|_{\vI,\vM}^2\nonumber\\
\leq 	& \textstyle  - {2\beta-\gamma\over2\beta} \|\vz_1-\vz_1^+-\vz_2+\vz_2^+\|^2 -  \|\vs_1-\vs_1^+-\vs_2+\vs_2^+\|_\vM^2.
\end{align}
\end{remark}

\begin{corollary}
\label{cor:CP}
When $f=0$ and $l^*=0$, we have 
\begin{align*}
      & \|\vz_1^+-\vz_2^+\|^2 + \|\vs_1^+-\vs_2^+\|_\vM^2 - \|\vz_1-\vz_2\|^2 -\|\vs_1-\vs_2\|_\vM^2  \\
\leq  & - \|\vz_1^+-\vz_1-\vz_2^++\vz_2\|^2 - \|\vs_1^+-\vs_1-\vs_2^++\vs_2\|_\vM^2,
\end{align*}
for any $\gamma$ and $\delta$ such that $\gamma\delta\|\vA\vA^\top\|\leq1$. 
It means that the Chambolle-Pock iteration is equivalent to a firmly nonexpansive operator under the norm $\|(\cdot,\cdot)\|_{\vI,\vM}$ when $\vM$ is positive definite. 
\end{corollary}

\begin{proof}
Letting $f=0$ and $l^*=0$ in~\eqref{eq:nonexpansive_cp} immediately gives the result. \qed
\end{proof}

The equivalence of the Chambolle-Pock iteration to a firmly nonexpansive operator under the norm defined by ${1\over \gamma}\|\vx\|^2-2\langle \vA\vx,\vs\rangle +{1\over\delta}\|\vs\|^2$ is shown in~\cite{peng2016coordinate,he2014convergence,7025841} by reformulating Chambolle-Pock as a proximal point algorithm applied on the KKT conditions. 
Here, we show the firmly nonexpansiveness of the Chambolle-Pock iteration for $(\vz,\vs)$ under a different norm defined by $\|(\cdot,\cdot)\|_{\vI,\vM}$. 
In fact, the Chambolle-Pock iteration is equivalent to DRS applied on the KKT conditions based on the connection between PD3O and Davis-Yin in Remark~\ref{remark1} and that Davis-Yin reduces to DRS without the Lipschitz continous operator.
The equivalence of the Chambolle-Pock iteration and DRS is also showed in~\cite{o2017equivalence} using $(\vz,\sqrt{\vM}\vs)$ under the standard norm instead of $(\vz,\vs)$ under the norm $\|(\cdot,\cdot)\|_{\vI,\vM}$, therefore the convergence condition of Chambolle-Pock becomes $\gamma\delta\|\vA\vA^\top\|\leq 1$, if we do not consider the convergence of the dual variable $\vs$. 

\begin{corollary}[$\alpha$-averageness of PAPC]\label{cor:PDFP2O}
When $g=0$ and $l^*=0$, $\PD$ reduces to the PAPC iteration, and it is $\alpha$-averaged with $\alpha={2\beta/(4\beta-\gamma)}$ when $\gamma<2\beta$ and $\gamma\delta\|\vA\vA^\top\|<1$.
\end{corollary}

In~\cite{chen2013primal}, the PAPC iteration is shown to be nonexpansive \textnormal{only} under a norm for $\cX\times \cS$ that is defined by $\sqrt{\|\vz\|^2+{\gamma\over \delta}\|\vs\|^2}$. Corollary~\ref{cor:PDFP2O} improves the result by showing that it is $\alpha$-averaged with $\alpha=2\beta/(4\beta-\gamma)$ under the norm $\|(\cdot,\cdot)\|_{\vI,\vM}$. 
This result appeared in~\cite{davis2015convergence} previously.

\begin{lemma}\label{lemma:optimal_solution}For any fixed point $(\vz^*,\vs^*)$ of $\PD$, $\prox_{\gamma g}(\vz^*)$ is an optimal solution to the optimization problem~\eqref{for:main_problem}. For any optimal solution $\vx^*$ of the optimization problem~\eqref{for:main_problem} satisfying $\vzero\in  \partial g(\vx^*)+ \nabla f(\vx^*)+ \vA^\top\partial h\square l(\vA\vx^*)$, we can find a fixed point $(\vz^*,\vs^*)$ of $\PD$ such that $\vx^*=\prox_{\gamma g}(\vz^*)$.
\end{lemma}
\begin{proof}
If $(\vz^*,\vs^*)$ is a fixed point of $\PD$, let $\vx^*=\prox_{\gamma g}(\vz^*)$. 
Then we have $\vzero= \vz^*-\vx^*+\gamma\nabla f(\vx^*)+ \gamma\vA^\top\vs^*$ from~\eqref{for:PD3O_iteration_c}, $\vz^*-\vx^*\in\gamma \partial g(\vx^*)$ from~\eqref{for:PD3O_iteration_a}, and $\vA\vx^*\in\partial h^*(\vs^*)+\nabla l^*(\vs^*)$ from~\eqref{for:PD3O_iteration_b} and~\eqref{for:PD3O_iteration_c}. Therefore, $\vzero\in \gamma \partial g(\vx^*)+\gamma \nabla f(\vx^*)+\gamma \vA^\top\partial h\square l(\vA\vx^*)$, i.e., $\vx^*$ is an optimal solution for the convex problem~\eqref{for:main_problem}.

If $\vx^*$ is an optimal solution for problem~\eqref{for:main_problem} such that $\vzero\in  \partial g(\vx^*)+ \nabla f(\vx^*)+ \vA^\top\partial h\square l(\vA\vx^*)$, there exist $\vq_g^*\in\partial g(\vx^*)$ and $\vq_{h}^*\in\partial h\square l(\vA\vx^*)$ such that $\vzero=\vq_g^* +\nabla f(\vx^*) + \vA^\top\vq_{h}^*$. Letting $\vz^*=\vx^*+\gamma \vq_g^*$ and $\vs^*=\vq_{h}^*$, we derive
$\vx=\vx^*$ from~\eqref{for:PD3O_iteration_a}, $\vs^+=\prox_{\delta h^*}(\vs^*-\delta\nabla l^*(\vs^*)+\delta\vA\vx^*)=\vs^*$ from~\eqref{for:PD3O_iteration_b}, and $\vz^+=\vx^*-\gamma\nabla f(\vx^*)-\gamma \vA^\top\vs^*=\vx^*+\gamma \vq_g^*=\vz^*$ from~\eqref{for:PD3O_iteration_c}. Thus $(\vz^*,\vs^*)$ is a fixed point of $\PD$. \qed
\end{proof}

\begin{theorem}[sublinear convergence rate]\label{thm:main}
Choose $\gamma$ and $\delta$ such that $\gamma<2\beta$ and Assumption~\ref{assum:1} is satisfied. Let $(\vz^*,\vs^*)$ be any fixed point of $\PD$, and $(\vz^k,\vs^k)_{k\geq0}$ is the sequence generated by PD3O. 
\begin{enumerate}
\item[1)] The sequence $(\|(\vz^k,\vs^k)-(\vz^*,\vs^*)\|_{\vI,\vM})_{k\geq 0}$ is monotonically nonincreasing.
\item[2)] The sequence $(\|\PD(\vz^k,\vs^k)-(\vz^k,\vs^k)\|_{\vI,\vM})_{k\geq 0}$ is monotonically nonincreasing and converges to 0.
\item[3)] We have the following convergence rate
\begin{align}
 \|\PD(\vz^k,\vs^k)-(\vz^k,\vs^k)\|_{\vI,\vM}^2\leq {2\beta\over 2\beta-\gamma } {\|(\vz^0,\vs^0)-(\vz^*,\vs^*)\|_{\vI,\vM}^2\over k+1} \label{eqn:rate_fp}
\end{align}
and 
\begin{align*}
\|\PD(\vz^k,\vs^k)-(\vz^k,\vs^k)\|_{\vI,\vM}^2=o\left({1\over k+1}\right).
\end{align*}
\item[4)] $(\vz^k,\vs^k)$ weakly converges to a fixed point of $\PD$. %and if $\cX$ has a finite dimension, then it is strongly convergent.
\end{enumerate}
\end{theorem}

\begin{proof}
1)  Let $(\vz_1,\vs_1)=(\vz^k,\vs^k)$ and $(\vz_2,\vs_2)=(\vz^*,\vs^*)$ in~\eqref{eq:average}, and we have 
\begin{align}
			& \|(\vz^{k+1},\vs^{k+1})-(\vz^*,\vs^*)\|_{\vI,\vM}^2 -\|(\vz^k,\vs^k)-(\vz^*,\vs^*)\|_{\vI,\vM}^2 \nonumber\\
\leq  & -{2\beta-\gamma\over2\beta}\|\PD(\vz^{k},\vs^{k})-(\vz^k,\vs^k)\|_{\vI,\vM}^2.\label{eq:average2}
\end{align}
Thus,  $\|(\vz^k,\vs^k)-(\vz^*,\vs^*)\|_{\vI,\vM}^2$ is monotonically decreasing as long as $\PD(\vz^{k},\vs^{k})-(\vz^k,\vs^k)\neq \vzero$ because $\gamma<2\beta$. 

2) Summing up~\eqref{eq:average2} from 0 to $\infty$, we have 
$$ \sum_{k=0}^\infty\|\PD(\vz^k,\vs^k)-(\vz^k,\vs^k)\|_{\vI,\vM}^2\leq {2\beta\over 2\beta-\gamma}\|(\vz^0,\vs^0)-(\vz^*,\vs^*)\|_{\vI,\vM}^2$$
and thus the sequence $(\|\PD(\vz^k,\vs^k)-(\vz^k,\vs^k)\|_{\vI,\vM})_{k\geq 0}$ converges to 0. 
Furthermore, the sequence $(\|\PD(\vz^k,\vs^k)-(\vz^k,\vs^k)\|_{\vI,\vM})_{k\geq 0}$ is monotonically nonincreasing because 
\begin{align*}
    &\|\PD(\vz^{k+1},\vs^{k+1})-(\vz^{k+1},\vs^{k+1})\|_{\vI,\vM}= \|\PD(\vz^{k+1},\vs^{k+1})-\PD(\vz^{k},\vs^{k})\|_{\vI,\vM} \\
  \leq & \|(\vz^{k+1},\vs^{k+1})-(\vz^{k},\vs^{k})\|_{\vI,\vM}=\|\PD(\vz^{k},\vs^{k})-(\vz^{k},\vs^{k})\|_{\vI,\vM}.
\end{align*}

3) The $o(1/k)$ convergence rate follows from~\cite[Theorem 1]{davis2014convergence}. 
\begin{align*}
\|\PD(\vz^k,\vs^k)-(\vz^k,\vs^k)\|_{\vI,\vM}^2 \leq &  {1\over k+1}\sum_{i=0}^k\|\PD(\vz^i,\vs^i)-(\vz^i,\vs^i)\|_{\vI,\vM}^2  \\
\leq &  {2\beta\over 2\beta-\gamma } {\|(\vz^0,\vs^0)-(\vz^*,\vs^*)\|_{\vI,\vM}^2\over k+1}.
\end{align*}

4) This follows from the Krasnosel'skii-Mann theorem~\cite[Theorem 5.14]{bauschke2011convex}.\qed
\end{proof}

\begin{remark}
The convergence of $\vz^k$ (or $\vx^k$) can be obtained under a weaker condition that $\vM$ is positive semidefinite, instead of positive definite if $\nabla l^*$ is a constant. 
In this case, the weak convergence of $(\vz^k,\sqrt{\vM}\vs^k)$ can be obtained under the standard norm. 
Furthermore, if we assume that $\cX$ is finite dimensional, we have the strong convergence of $\vz^k$ and thus $\vx^k$ because the proxmial operator in~\eqref{for:PD3O_iteration_a} is nonexpansive.
\end{remark}

\begin{remark}
Since the operator $\PD$ is $\alpha$-averaged with $\alpha={2\beta/(4\beta-\gamma)}$, 
the relaxed iteration $(\vz^{k+1},\vs^{k+1})=\theta_k\PD(\vz^k,\vs^k) + (1-\theta_k)(\vz^k,\vs^k)$ still converges if $\sum_{k=0}^\infty \theta_k (2-\gamma/(2\beta)-\theta_k) =\infty$.
\end{remark}

\subsection{$O(1/k)$-ergodic convergence for the general convex case}\label{sec:primal_dual}
In this subsection, let 
\begin{align}
\cL(\vx,\vs)=f(\vx)+g(\vx)+\langle\vA\vx,\vs\rangle- h^*(\vs)-l^*(\vs).
\end{align}
If $\cL$ has a saddle point, there exists $(\vx^*,\vs^*)$ such that 
$$\cL(\vx^*,\vs)\leq \cL(\vx^*,\vs^{*})\leq \cL(\vx,\vs^{*}),~\forall (\vx,\vs)\in\cX\times\cS.$$

Then, we consider the quantity $\cL(\vx^k,\vs)-\cL(\vx,\vs^{k+1})$ for all $(\vx,\vs)\in\cX\times \cS$. 
If $\cL(\vx^k,\vs)-\cL(\vx,\vs^{k+1})\leq 0$ for all $(\vx,\vs)$, then we have  
$$\cL(\vx^k,\vs)\leq \cL(\vx^k,\vs^{k+1})\leq \cL(\vx,\vs^{k+1}),$$
which means that $(\vx^k,\vs^{k+1})$ is a saddle point of $\cL$.

\begin{theorem}
Let $\gamma\leq \beta$ and Assumption~\ref{assum:1} be satisfied. The sequence $(\vx^k,\vz^k,\vs^k)_{k\geq 0}$ is generated by PD3O. Define
\begin{align*}
 \overline\vx^k = {1\over k+1}\sum_{i=0}^k\vx^i,\qquad \overline\vs^{k+1} = {1\over k+1}\sum_{i=0}^k\vs^{i+1}.
\end{align*}
Then we have 
\begin{align}
      \cL(\bar\vx^k,\vs)-\cL(\vx,\bar\vs^{k+1}) 
\leq & {1\over 2(k+1)\gamma}\|(\vz,\vs)-(\vz^0,\vs^0)\|_{\vI,\vM}^2.
\end{align}
for any $(\vx,\vs)\in\cX\times\cS$ and $\vz=\vx-\gamma\nabla f(\vx)-\gamma\vA^\top\vs$.
\end{theorem}
\begin{proof}
First, we provide an upper bound for $\cL(\vx^k,\vs)-\cL(\vx,\vs^{k+1})$, and to do this, we consider upper bounds for $\cL(\vx^k,\vs)-\cL(\vx^k,\vs^{k+1})$ and $\cL(\vx^k,\vs^{k+1})-\cL(\vx,\vs^{k+1})$.

For any $\vx\in\cX$, we have 
\begin{align}
		 & \cL(\vx^k,\vs^{k+1})-\cL(\vx,\vs^{k+1}) \nonumber\\
= 	 & f(\vx^k)+g(\vx^k)+\langle \vA\vx^k,\vs^{k+1}\rangle-f(\vx)-g(\vx)-\langle\vA\vx,\vs^{k+1}\rangle \nonumber\\
\leq & \langle \nabla f(\vx^k),\vx^k-\vx\rangle -{(\beta/2)}\|\nabla f(\vx^k)-\nabla f(\vx)\|^2 \nonumber\\
     & + \gamma^{-1}\langle \vx-\vx^k,\vx^k-\vz^k\rangle +\langle \vx^k-\vx,\vA^\top\vs^{k+1}\rangle \nonumber\\
=    & \gamma^{-1}\langle \vz^{k+1}-\vz^{k},\vx-\vx^k\rangle-{(\beta/2)}\|\nabla f(\vx^k)-\nabla f(\vx)\|^2. \label{ergodic_a}
\end{align}
The inequality comes from~\eqref{lemma1a} (with $\vx_1$ and $\vx_2$ being $\vx$ and $\vx^k$, respectively) and~\eqref{for:PD3O_iteration_a}. 
The last equality holds because of~\eqref{for:PD3O_iteration_c}. 

On the other hand, combing~\eqref{for:PD3O_iteration_b} and~\eqref{for:PD3O_iteration_c}, we obtain 
\begin{align*}
\vs^{k+1} &=  \prox_{{\delta} h^*} ((\vI-\gamma\delta \vA\vA^\top)\vs^k-{\delta}\nabla l^*(\vs^k) + {\delta}\vA\left(2\vx^k-\vz^k-{\gamma}\nabla f(\vx^k)\right)) \\
&=        \prox_{{\delta} h^*} (\vs^k+\gamma\delta \vA\vA^\top(\vs^{k+1}-\vs^k)-{\delta}\nabla l^*(\vs^k) + {\delta}\vA\left(\vx^k+\vz^{k+1}-\vz^k\right)).
\end{align*}
Therefore, for any $\vs\in\cS$, the following inequality holds.
\begin{align}
     & h^*(\vs^{k+1})-h^*(\vs) \nonumber\\
\leq & \langle \gamma^{-1}\vM(\vs^{k+1}-\vs^k)+\nabla l^*(\vs^k) - \vA\left(\vx^k+\vz^{k+1}-\vz^k\right),\vs-\vs^{k+1}	\rangle \nonumber\\
=    & \gamma^{-1}\langle \vs^{k+1}-\vs^k,\vs-\vs^{k+1}	\rangle_\vM +\langle \nabla l^*(\vs^k),\vs-\vs^{k+1}\rangle \nonumber\\
     &  -\langle \vz^{k+1}-\vz^k,\vA^{\top}(\vs-\vs^{k+1})	\rangle  -\langle \vA\vx^k,\vs-\vs^{k+1}	\rangle.  \label{inequal_h}
\end{align}
Additionally, from~\eqref{lemma1b}, we have 
\begin{align}
l^*(\vs^{k+1}) \leq l^*(\vs^k) + \langle \nabla l^*(\vs^k),\vs^{k+1}-\vs^k\rangle + (2\beta)^{-1}\|\vs^{k+1}-\vs^k\|_\vM^2,~\label{dual_l1}
\end{align}
and the convexity of $l^*$ gives
\begin{align}
l^*(\vs) \geq l^*(\vs^k) + \langle \nabla l^*(\vs^k),\vs-\vs^k\rangle. ~\label{dual_l2}
\end{align}
Combining~\eqref{inequal_h},~\eqref{dual_l1}, and~\eqref{dual_l2}, we derive 
\begin{align}
		 & \cL(\vx^k,\vs)-\cL(\vx^k,\vs^{k+1}) \nonumber\\
=    & \langle \vA\vx^k,\vs-\vs^{k+1}\rangle +h^*(\vs^{k+1})+l^*(\vs^{k+1})-h^*(\vs)-l^*(\vs)\nonumber\\
\leq & \gamma^{-1}\langle \vs-\vs^{k+1},\vs^{k+1}-\vs^k\rangle_\vM-\langle \vz^{k+1}-\vz^k,\vA^\top(\vs-\vs^{k+1})\rangle \nonumber\\
     & +(2\beta)^{-1}\|\vs^{k+1}-\vs^k\|_\vM^2.  \label{ergodic_b} 
\end{align}

Then, adding~\eqref{ergodic_a} and~\eqref{ergodic_b} together gives 
\begin{align*}
     & \cL(\vx^k,\vs)-\cL(\vx,\vs^{k+1}) \\
\leq & \gamma^{-1}\langle \vs-\vs^{k+1},\vs^{k+1}-\vs^k\rangle_\vM-\langle \vz^{k+1}-\vz^k,\vA^\top(\vs-\vs^{k+1})\rangle \\
     &  +\gamma^{-1}\langle \vz^{k+1}-\vz^{k},\vx-\vx^k\rangle-{(\beta/2)}\|\nabla f(\vx^k)-\nabla f(\vx)\|^2 \\
		 & +(2\beta)^{-1}\|\vs^{k+1}-\vs^k\|_\vM^2\\
=    & \gamma^{-1}\langle \vs-\vs^{k+1},\vs^{k+1}-\vs^k\rangle_\vM +\gamma^{-1}\langle \vz^{k+1}-\vz^{k},\vz-\vz^{k+1}\rangle\\
     & +\langle \vz^{k+1}-\vz^{k},\nabla f(\vx)-\nabla f(\vx^k)\rangle-{(\beta/2)}\|\nabla f(\vx^k)-\nabla f(\vx)\|^2\\
		 & +(2\beta)^{-1}\|\vs^{k+1}-\vs^k\|_\vM^2\\
\leq & (2\gamma)^{-1}(\|\vs-\vs^k\|_\vM^2-\|\vs-\vs^{k+1}\|_\vM^2-\|\vs^k-\vs^{k+1}\|_\vM^2) \\
     & +(2\gamma)^{-1}(\|\vz-\vz^{k}\|^2-\|\vz-\vz^{k+1}\|^2-\|\vz^k-\vz^{k+1}\|^2) \\  
		 & +(2\beta)^{-1}\|\vz^k-\vz^{k+1}\|^2+(2\beta)^{-1}\|\vs^{k+1}-\vs^k\|_\vM^2.
\end{align*}
That is 
\begin{align}
\cL(\vx^k,\vs)-\cL(\vx,\vs^{k+1}) \leq & (2\gamma)^{-1}\|(\vz^k,\vs^k)-(\vz,\vs)\|_{\vI,\vM}^2\nonumber\\
 & -(2\gamma)^{-1}\|(\vz^{k+1},\vs^{k+1})-(\vz,\vs)\|_{\vI,\vM}^2 \label{eqn:gap_conv_a}\\
 &  -((2\gamma)^{-1}-(2\beta)^{-1})\|(\vz^k,\vs^k)-(\vz^{k+1},\vs^{k+1})\|_{\vI,\vM}^2. \nonumber
\end{align}
When $\gamma\leq \beta$, we have 
\begin{align}
      \cL(\vx^k,\vs)-\cL(\vx,\vs^{k+1}) 
\leq & (2\gamma)^{-1}\|(\vz,\vs)-(\vz^{k},\vs^k)\|_{\vI,\vM}^2 \nonumber\\
     & -(2\gamma)^{-1}\|(\vz,\vs)-(\vz^{k+1},\vs^{k+1})\|_{\vI,\vM}^2. \label{conv:gap_one_step}
\end{align}
Using the definition of $(\bar\vx^k,\bar\vs^{k+1})$ and the Jensen inequality, it follows that 
\begin{align}
      \cL(\bar\vx^k,\vs)-\cL(\vx,\bar\vs^{k+1})  
\leq & {1\over{k+1}}\sum_{i=0}^k \cL(\vx^i,\vs)-\cL(\vx,\vs^{i+1}) \nonumber\\
\leq & {1\over 2(k+1)\gamma}\|(\vz,\vs)-(\vz^0,\vs^0)\|_{\vI,\vM}^2.
\end{align}
This proves the desired result. \qed
\end{proof}

The $O(1/k)$-ergodic convergence rate proved in~\cite{Drori2015209} for PAPC is on a different sequence $\left({1\over k+1}\sum_{i=1}^{k+1}\vx^{i},{1\over k+1}\sum_{i=1}^{k+1}\vs^{i}=\bar\vs^{k+1}\right)$.

\subsection{Linear convergence rate for special cases}\label{sec:linear}
In this subsection, we provide some results on the linear convergence rate of PD3O with additional assumptions. 
For simplicity, let $(\vz^*,\vs^*)$ be a fixed point of $\PD$ and $\vx^*=\prox_{\gamma g}(\vz^*)$. 
In addition, we let $\langle \vs-\vs^*,\vq_{h^*}(\vs)-\vq_{h^*}^*\rangle\geq \tau_{h^*}\|\vs-\vs^*\|^2_\vM$ and $\langle \vs-\vs^*,\nabla l^*(\vs)-\nabla l^*(\vs^*)\rangle\geq \tau_{l^*}\|\vs-\vs^*\|^2_\vM$ for any $\vs\in\cS$, $\vq_{h^*}(\vs)\in\partial h^*(\vs)$, and $\vq_{h^*}^*\in\partial h^*(\vs^*)$. 
Then $\vM^{-1}\partial h^*$ and $\vM^{-1}\nabla l^*$ are restricted $\tau_{h^*}$-strongly monotone and restricted $\tau_{l^*}$-strongly monotone under the norm defined by $\|(\cdot,\cdot)\|_{\vI,\vM}$, respectively. 
Here we allow that $\tau_{h^*}=0$ and $\tau_{l^*}=0$ for merely monotone operators.
Similarly, we let $\langle \vx-\vx^*,\vq_g(\vx)-\vq_g^*\rangle\geq \tau_g\|\vx-\vx^*\|^2$ and $\langle  \vx-\vx^*,\nabla f(\vx)-\nabla f(\vx^*)\rangle \geq \tau_f\|\vx-\vx^*\|^2$ for any $\vx\in\cX$, $\vq_g(\vx)\in\partial g(\vx)$, and $\vq_g^*\in\partial g(\vx^*)$.

\begin{theorem}[linear convergence rate]
If $g$ has a $L_g$-Lipschitz continuous gradient, i.e., 
$$\|\nabla g(\vx)-\nabla g(\vy)\|\leq L_g\|\vx-\vy\|$$ and $(\vz^+,\vs^+)=\PD(\vz,\vs)$, then under Assumption~\ref{assum:1}, we have
\begin{align*}
	    \textstyle \|\vz^+-\vz^*\|^2+\left(1+2\gamma\tau_{h^*}\right)\|\vs^+-\vs^*\|_\vM^2 \leq \rho\left(\|\vz-\vz^*\|^2+\left(1+2\gamma\tau_{h^*}\right)\|\vs-\vs^*\|_\vM^2 \right)
\end{align*}
where 
\begin{align}\label{eqn:rho}
\textstyle \rho = \max\left( {{1-\left(2\gamma-{\gamma^2\over\beta}\right)\tau_{l^*}} \over 1 +2\gamma\tau_{h^*}},  1- {\left(\left(2\gamma-{\gamma^2\over\beta}\right)\tau_f+2\gamma\tau_g\right)\over 1+\gamma L_g}\right).
\end{align}
When, in addition, $\gamma<2\beta$, $\tau_{h^*}+\tau_{l^*}>0$, and $\tau_f+\tau_g>0$, we have that $\rho<1$ and the algorithm PD3O converges linearly.
\end{theorem}

\begin{proof}
Equation~\eqref{eqn:fundamental_1} in Lemma~\ref{lemma:fundamental} with $(\vz_1,\vs_1)=(\vz,\vs)$ and $(\vz_2,\vs_2)=(\vz^*,\vs^*)$ gives 

\begin{align*}
			&2\textstyle \gamma\langle \vs^+-\vs^*,\vq_{h^*}(\vs^+)-\vq_{h^*}^*+\nabla l^*(\vs)-\nabla l^*(\vs^*)\rangle \\
			& + 2\gamma\langle \vx-\vx^*,\vq_g(\vx)-\vq_g^*+\nabla f(\vx)-\nabla f(\vx^*)\rangle \\
	=   &\textstyle \|(\vz,\vs)-(\vz^*,\vs^*)\|_{\vI,\vM}^2 - \|(\vz^+,\vs^+)-(\vz^*,\vs^*)\|_{\vI,\vM}^2 - \|(\vz,\vs)-(\vz^+,\vs^+)\|_{\vI,\vM}^2 \\
	    &\textstyle + 2\gamma\langle \vz-\vz^+,\nabla f(\vx)-\nabla f(\vx^*)\rangle.
\end{align*}
Rearranging it, we have 
\begin{align*}
	    & \|(\vz^+,\vs^+)-(\vz^*,\vs^*)\|_{\vI,\vM}^2 \\
	=   & \|(\vz,\vs)-(\vz^*,\vs^*)\|_{\vI,\vM}^2 - \|(\vz,\vs)-(\vz^+,\vs^+)\|_{\vI,\vM}^2  + 2\gamma\langle \vz-\vz^+,\nabla f(\vx)-\nabla f(\vx^*)\rangle\\
			& - 2\gamma\langle \vs^+-\vs^*,\vq_{h^*}(\vs^+)-\vq_{h^*}^*+\nabla l^*(\vs)-\nabla l^*(\vs^*)\rangle \\
			& - 2\gamma\langle \vx-\vx^*,\vq_g(\vx)-\vq_g^*+\nabla f(\vx)-\nabla f(\vx^*)\rangle.
\end{align*}

The Cauchy-Schwarz inequality gives us 
\begin{align*}
 2\gamma\langle \vz-\vz^+,\nabla f(\vx)-\nabla f(\vx^*)\rangle \leq \|\vz-\vz^+\|^2+\gamma^2\|\nabla f(\vx)-\nabla f(\vx^*)\|^2,
\end{align*}
and the cocoercivity of $\nabla f$ shows
\begin{align*}
\textstyle -{\gamma^2\over \beta}\langle  \vx-\vx^*,\nabla f(\vx)-\nabla f(\vx^*)\rangle\leq -\gamma^2\|\nabla f(\vx)-\nabla f(\vx^*)\|^2.
\end{align*}
Combining the previous two inequalities, we have 
\begin{align*}
	    &  -\|\vz-\vz^+\|^2  + 2\gamma\langle \vz-\vz^+,\nabla f(\vx)-\nabla f(\vx^*)\rangle\\
			& - 2\gamma\langle \vx-\vx^*,\vq_g(\vx)-\vq_g^*+\nabla f(\vx)-\nabla f(\vx^*)\rangle\\
\leq  & \textstyle  - \left(2\gamma-{\gamma^2\over\beta}\right)  \langle  \vx-\vx^*,\nabla f(\vx)-\nabla f(\vx^*)\rangle - 2\gamma\langle \vx-\vx^*,\vq_g(\vx)-\vq_g^*\rangle \\
\leq  & \textstyle  - \left(\left(2\gamma-{\gamma^2\over\beta}\right)\tau_f+2\gamma \tau_g\right)  \|\vx-\vx^*\|^2.
\end{align*}
Similarly, we derive 
\begin{align*}
	    & - \|\vs-\vs^+\|_\vM^2 - 2\gamma\langle \vs^+-\vs^*,\vq_{h^*}(\vs^+)-\vq_{h^*}^*+\nabla l^*(\vs)-\nabla l^*(\vs^*)\rangle\\
=     & \textstyle  - \|\vs-\vs^+\|_\vM^2- 2\gamma\langle \vs^+-\vs,\nabla l^*(\vs)-\nabla l^*(\vs^*)\rangle \\
      &  -2\gamma\langle \vs-\vs^*,\nabla l^*(\vs)-\nabla l^*(\vs^*)\rangle - 2\gamma\langle \vs^+-\vs^*,\vq_{h^*}(\vs^+)-\vq_{h^*}^*\rangle \\
\leq  & \textstyle  - \left(2\gamma-{\gamma^2\over\beta}\right)  \langle  \vs-\vs^*,\nabla l^*(\vs)-\nabla f(\vs^*)\rangle - 2\gamma\langle \vs^+-\vs^*,\vq_{h^*}(\vs^+)-\vq_{h^*}^*\rangle \\
\leq  & \textstyle  - \left(2\gamma-{\gamma^2\over\beta}\right)\tau_{l^*}  \|\vs-\vs^*\|_\vM^2-2\gamma\tau_{h^*}\|\vs^+-\vs^*\|^2_\vM.
\end{align*}

Therefore, we have 
\begin{align*}
	    &\textstyle  \|(\vz^+,\vs^+)-(\vz^*,\vs^*)\|_{\vI,\vM}^2 \\
\leq  &\textstyle  \|(\vz,\vs)-(\vz^*,\vs^*)\|_{\vI,\vM}^2 - \left(\left(2\gamma-{\gamma^2\over\beta}\right)\tau_f+2\gamma\tau_g\right)\|\vx-\vx^*\|^2 \\
      &\textstyle  - \left(2\gamma-{\gamma^2\over\beta}\right)\tau_{l^*}\|\vs-\vs^*\|_\vM^2-2\gamma\tau_{h^*} \|\vs^+-\vs^*\|_\vM^2.
\end{align*}
Finally 
\begin{align*}
	    &\textstyle  \|\vz^+-\vz^*\|^2+\left(1+2\gamma\tau_{h^*}\right)\|\vs^+-\vs^*\|_\vM^2 \\
\leq  &\textstyle  \|\vz-\vz^*\|^2+\left(1- \left(2\gamma-{\gamma^2\over\beta}\right)\tau_{l^*}\right)\|\vs-\vs^*\|_\vM^2 - \left(\left(2\gamma-{\gamma^2\over\beta}\right)\tau_f+2\gamma\tau_g\right)\|\vx-\vx^*\|^2 \\
\leq  &\textstyle  \|\vz-\vz^*\|^2+\left(1- \left(2\gamma-{\gamma^2\over\beta}\right)\tau_{l^*}\right)\|\vs-\vs^*\|_\vM^2 - {\left(\left(2\gamma-{\gamma^2\over\beta}\right)\tau_f+2\gamma\tau_g\right)\over 1+\gamma L_g}\|\vz-\vz^*\|^2\\
\leq  &\textstyle  \rho \left(\|\vz-\vz^*\|^2+(1+2\gamma\tau_{h^*})\|\vs-\vs^*\|_\vM^2\right),
\end{align*}
with $\rho$ defined in~\eqref{eqn:rho}. \qed
\end{proof}

\section{Numerical experiments} \label{sec:numerical}

In this section, we compare PD3O with PDFP, AFBA, and Condat-Vu in solving the fused lasso problem~\eqref{eqn:fusedlasso}. 
Note that many other algorithms can be applied to solve this problem. 
For example, the proximal mapping of $\mu_2\|\vD \vx\|_1$ can be computed exactly and quickly~\cite{condat2013direct}, and Davis-Yin can be applied directly without using the primal-dual algorithms. 
Even for primal-dual algorithms, there are accelerated versions available~\cite{chambolle2016ergodic}.
However, it is not the focus of this paper to compare PD3O with all existing algorithms for solving the fused lasso problem. 
The numerical experiment in this section is used to validate and demonstrate the advantages of PD3O over other existing unaccelerated primal-dual algorithms: PDFP, AFBA, and Condat-Vu. 
More specifically, we would like to show the advantage of using a larger stepsize. 

Here, only the number of iterations is recorded because the computational cost for each iteration is very close for all four algorithms. 
In this example, the proximal mapping of $\mu_1\|\vx\|_1$ is easy to compute, and the additional cost of one proximal mapping in PDFP can be ignored.
The code for all the comparisons in this section is available at~\href{http://github.com/mingyan08/PD3O}{http://github.com/mingyan08/PD3O}.

We use the same setting as~\cite{chen2016primal}. 
Let $n=500$ and $p=10000$. 
$\vA$ is a random matrix whose elements follow the standard Gaussian distribution, and $\vb$ is obtained by adding independent and identically distributed Gaussian noise with variance 0.01 onto $\vA\vx$.
For the parameters, we set $\mu_1=20$ and $\mu_2=200$.

\begin{figure}[!h]
\begin{center}
\includegraphics[width=0.49\textwidth]{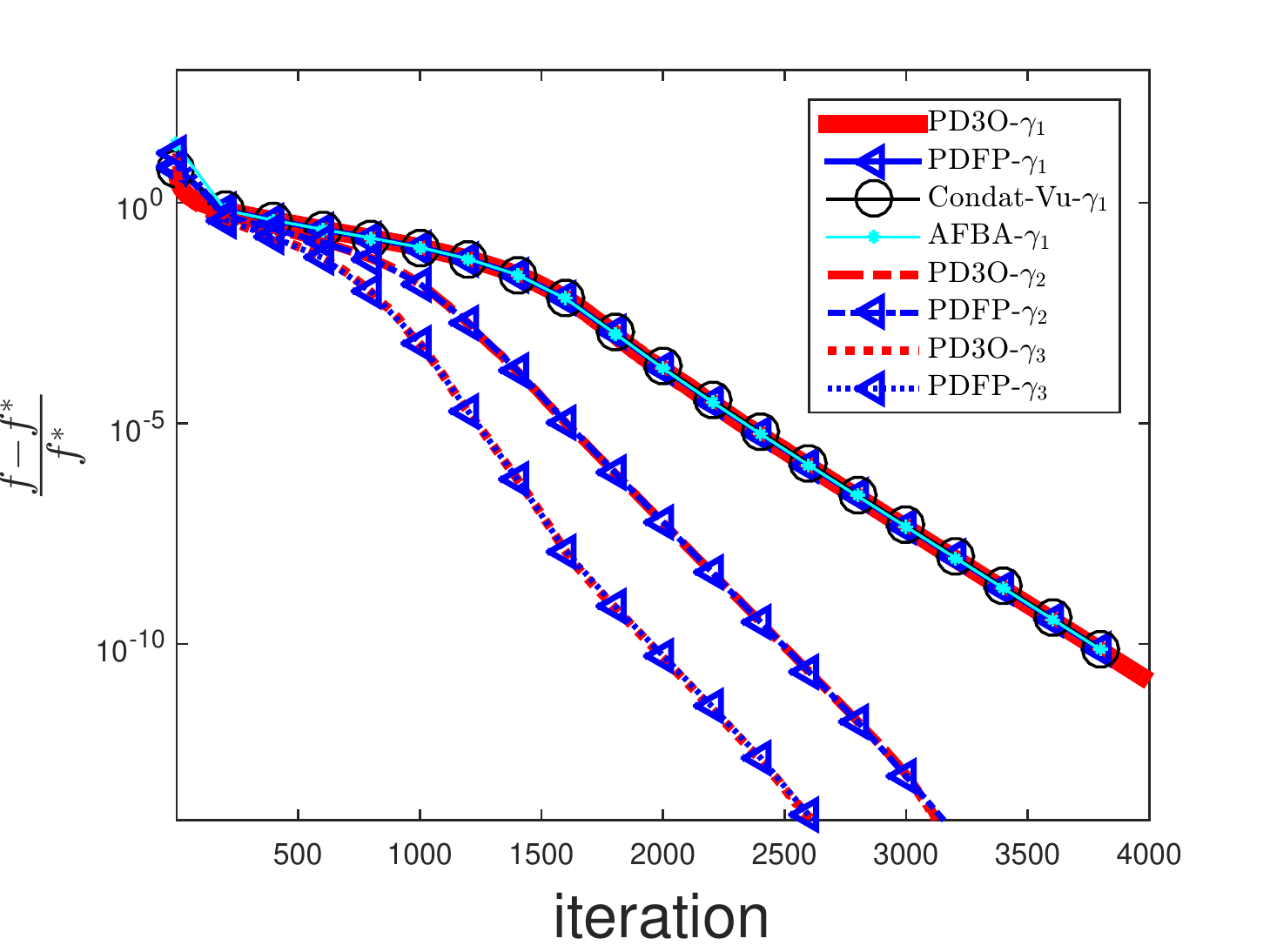}\includegraphics[width=0.49\textwidth]{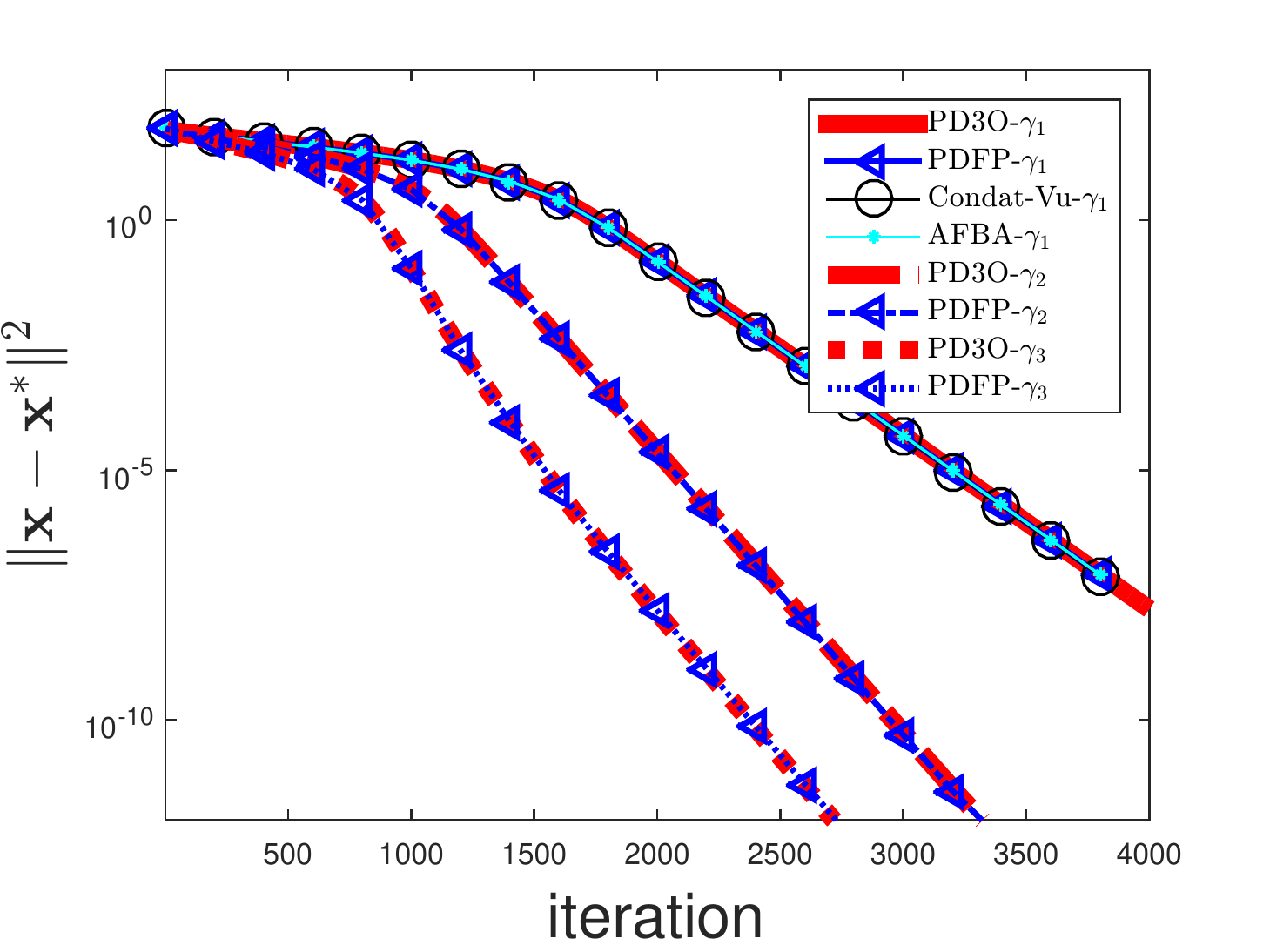}
\includegraphics[width=0.49\textwidth]{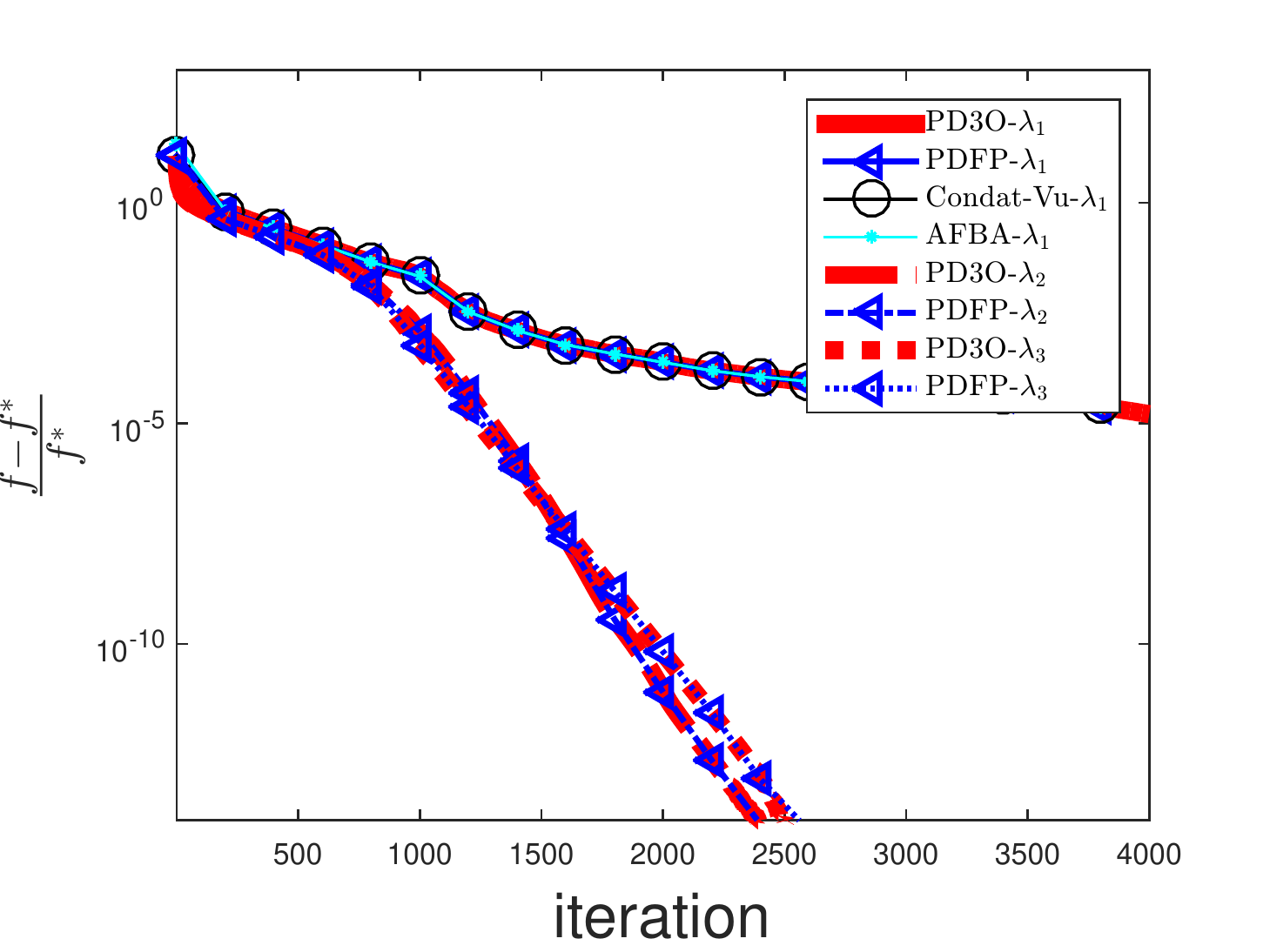}\includegraphics[width=0.49\textwidth]{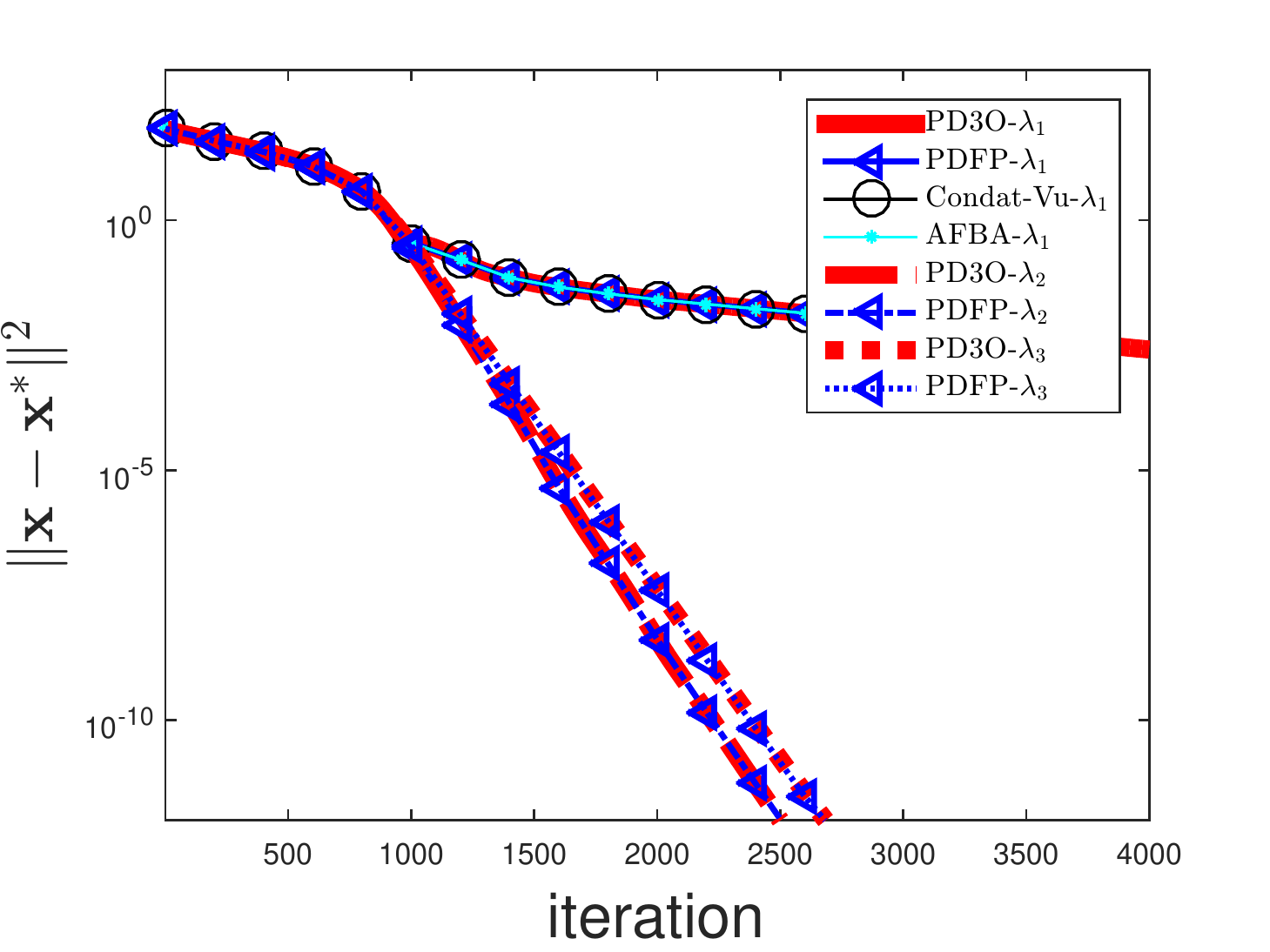}
\caption{The comparison of four algorithms (PD3O, PDFP, Condat-Vu, and AFBA) on the fused lasso problem in terms of primal objective values and the distances of $\vx^k$ to the optimal solution $\vx^*$ with respect to iteration numbers. In the top figures, we fix $\lambda=1/8$ and let $\gamma=\beta,~1.5\beta,~1.99\beta$. In the bottom figures, we fix $\gamma=1.9\beta$ and let $\lambda = 1/80,~1/8,~1/4$. PD3O and PDFP perform better than Condat-Vu and AFBA because they have larger ranges for acceptable parameters and choosing large numbers for both parameters makes the algorithm converge fast. Note PD3O has a slightly lower per-iteration complexity than PDFP.}%
\label{fig:flasso}%
\end{center}
\end{figure}

We would like to compare the four algorithms with different parameters, and the result may guide us in choosing parameters for these algorithms in other applications. 
Here we let $\lambda=\gamma\delta$.
Then recall that the parameters for Condat-Vu and AFBA have to satisfy $\lambda \left(2-2\cos ({p-1\over p}\pi)\right)+\gamma/(2\beta) \leq 1$ and $\lambda \left(1-\cos ({p-1\over p}\pi)\right)+\sqrt{\lambda (1-\cos ({p-1\over p}\pi))}/2+\gamma/(2\beta) \leq 1$, respectively, and those for PD3O and PDFP have to satisfy $\lambda \left(2-2\cos ({p-1\over p}\pi)\right) \leq 1$ and $\gamma< 2\beta$.
Firstly, we fix $\lambda=1/8$ and let $\gamma=\beta,~1.5\beta,~1.99\beta$. 
For Condat-Vu and AFBA, we only show $\gamma=\beta$ because $\gamma=1.5\beta$ and $1.99\beta$ do not satisfy their convergence conditions.
The primal objective values and the distances of $\vx^k$ to the optimal solution $\vx^*$ for these algorithms after each iteration are compared in~Fig.~\ref{fig:flasso} (Top). 
The optimal objective value $f^*$ is obtained by running PD3O for 20,000 iterations. 
The results show that all four algorithms have very close performance when they converge ($\gamma=\beta$). 
Both PD3O and PDFP converge faster with a larger stepsize $\gamma$. 
In addition, the figure shows that the speed of all four algorithms almost linearly depends on the primal stepsize $\gamma$, i.e., increasing $\gamma$ by two reduces the number of iterations by half.

Then we fix $\gamma=1.9\beta$ and let $\lambda = 1/80,~1/8,~1/4$. The objective values and the distances to the optimal solution for these algorithms after each iteration are compared in~Fig.~\ref{fig:flasso} (Bottom). 
Again, we can see that the performances for these three algorithms are very close when they converge ($\lambda=1/80$), and PD3O is slightly better than PDFP in terms of the number of iterations and the per-iteration complexity. 
However, when $\lambda$ changes from $1/8$ to $1/4$, there is no improvement in the convergence, and when the number of iterations is more than 2000, the performance of $\lambda=1/4$ is even worse than that of $\lambda=1/8$. 
This result also suggests that it is better to choose a slightly large $\lambda$  and the increase in $\lambda$ does not bring too much advantage if $\lambda$ is large enough (at least in this experiment). 
Both experiments demonstrate the effectiveness of having a large range of acceptable parameters.

\section{Conclusion}\label{sec:conclusion}

In this paper, we proposed a primal-dual three-operator splitting scheme PD3O for minimizing $f(\vx)+g(\vx)+h\square l(\vA\vx)$. 
It has primal-dual algorithms Chambolle-Pock and PAPC for minimizing the sum of two functions as special cases. 
Comparing to the three existing primal-dual algorithms PDFP, AFBA, and Condat-Vu for minimizing the sum of three functions, PD3O has the advantages from all these algorithms: a low per-iteration complexity and a large range of acceptable parameters ensuring the convergence. 
In addition, PD3O reduces to Davis-Yin when $\vA=\vI$ and special parameters are chosen.
The numerical experiments show the effectiveness and efficiency of PD3O. 
We left the acceleration and other modifications of PD3O as future work.

\section*{Acknowledgments.}
We would like to thank the anonymous reviewers for their helpful comments.

\bibliographystyle{spmpsci} 
\bibliography{PM3O} 

\end{document}